\documentclass[11pt,english]{article}%

\usepackage[letterpaper,top=2cm,bottom=2.5cm,left=3cm,right=3cm,marginparwidth=1.75cm]{geometry}

\usepackage{authblk}
\usepackage{orcidlink}

\usepackage[utf8]{inputenc}
\usepackage[english]{babel}

\usepackage[backend=biber]{biblatex}
\usepackage{graphicx}%
\usepackage{multirow}%
\usepackage{amsmath,amssymb,amsfonts}%
\usepackage{amsthm}%
\usepackage[title]{appendix}%
\usepackage{xcolor}%
\usepackage{textcomp}%
\usepackage{manyfoot}%
\usepackage{booktabs}%
\usepackage{algorithm}%
\usepackage{algorithmicx}%
\usepackage{algpseudocode}%
\usepackage{listings}%
\usepackage{mathrsfs}
\newtheorem{theorem}{Theorem}

\newtheorem{proposition}[theorem]{Proposition}%

\newtheorem{lemma}[theorem]{Lemma}
\newtheorem{remark}{Remark}%
\newtheorem{definition}{Definition}%

\providecommand{\keywords}[1]
{
  \small{\textit{Keywords:}} #1
}
\providecommand{\MSC}[1]
{
  \small{\textit{MSC Classification:}} #1
}
\usepackage{comment}
\usepackage[normalem]{ulem}
\newtheorem{thmy}{Theorem}

\begin{document}

\title{Puzzle Pieces, Bi-accessibility, and Connectivity of the Julia Set for Generalized Blaschke Products}

\author{{Melida Carranza Trejo}\orcidlink{0000-0001-5223-1300}}
\author{{M\'onica Moreno Rocha}\orcidlink{0000-0003-3816-4425}}

\affil{Centro de Investigaci\'on en Matem\'aticas, Guanajuato, 36023, Guanajuato,  M\'exico\\ (melida.carranza@cimat.mx, mmoreno@cimat.mx)}

\maketitle

\abstract{We study the dynamics of a parametric family of rational functions of odd degree, where each function is a generalized Blaschke product that maps the unit circle onto itself. 
The action of the Blaschke product restricted to the unit  circle defines a circle map, and the parameter space of the family exhibits Arnold tongues. As the parameter varies over an Arnold tongue, the action of the circle map changes from a diffeomorphism to a non-injective endomorphisms. Using a combinatorial study of puzzle pieces, we show that for adjacent parameters inside the Arnold tongues, there exist \emph{bi-accessible} repelling cycles. This topological feature enables us to exclude the presence of multiply connected Fatou components whenever Herman rings are absent. As a result, we obtain a complete characterization of the connectivity of the Julia set for each member of the parametric family.}

\vspace{0.2cm}

\keywords{Blaschke product, Julia set, Arnold tongues, connectivity}

\MSC{37F10, 37F20, 37F46}

\section{Introduction}\label{sec:intro}
The study of the connectivity of Julia sets is a central theme in holomorphic dynamics and, under certain conditions, it may be closely related to the behavior of critical orbits. For example, consider the iteration of a complex polynomial of degree at least 2: if the orbits of every finite critical point remains bounded, then its Julia set is connected, \cite{Milnor2011}. For the case of a rational function $f:\widehat{\mathbb{C}}\to \widehat{\mathbb{C}}$ of degree $d\geq 2$, its Julia set, $\mathcal{J}(f)$, is a compact set of the Riemann sphere, thus, $\mathcal{J}(f)$ is connected if and only if all its Fatou components are simply connected. Establishing the connectivity of every Fatou component is far more intricate in the rational case, since either Herman rings may occur, \cite{Herman1979}, or an immediate basin (either attracting or parabolic) may contain more than one critical point, which could result in a component of infinite connectivity (see for example \cite[Propositon 3.1]{Canela2020}).

In this paper, we analyze the connectivity of the Julia set for a specific family of rational maps known as generalized Blaschke products. While classical finite Blaschke products are endomorphisms of the unit disk $\mathbb{D}$ that map the unit circle $\mathbb{S}^1$ to itself, we consider a broader family $B_{c,t}$ defined by:
\begin{equation}
    B_{c,t}(z) = e^{2\pi i t} z^{d+1} \left( \frac{z-c}{1-\overline{c}z} \right)^d,
    \label{eq:intro_blaschke}
\end{equation}
where $d \in \mathbb{N}$, $c \in \mathbb{C}$ and $t\in\mathbb{R}$. These rational maps have degree $2d+1$ and hence, they have $4d$ critical points (counting multiplicity) with two of them at $0$ and $\infty$, which are also fixed under $B_a$ for all $a\in \mathbb{C}$. And once more, each $B_a$ sends $\mathbb{S}^1$ onto itself, although its is no longer backward invariant. Through a conformal conjugacy (see Lemma~\ref{lemma:conjugacy-t-to-0}) we can reduce the study of the bi-parametric family $B_{c,t}$ to a single complex parameter family, simply denoted as $B_a$.

The dynamics of $B_a$ is highly influenced by the position of its critical points relative to $\mathbb{S}^1$. We identify three distinct dynamical regions based on the modulus $|a|$: a region with simple dynamics, where the Julia set is the circle itself; an endomorphism region where the restriction $B_a|_{\mathbb{S}^1}$ acts as a degree-one circle map with critical points; and a homeomorphism or diffeomorphism region where the map on the circle is invertible. This behavior allows us to employ techniques from the theory of circle maps, specifically the concept of Arnold tongues, to organize the parameter space.

A brief overview of the theory of circle maps and holomorphic dynamics can be found in Section~\ref{sec:preliminaries}. In Section \ref{sec:TheBlaschke}, we provide a study  of the critical orbits and the parameter space of the family $B_a$. We show in Theorem~\ref{thm:rotation-domains} that the family $B_a$ does not realize  Siegel disks and that Herman rings cannot occur for parameters satisfying $|a| \le 2d+1$. Furthermore, we establish in Theorem~\ref{thm:super-attracting-basins} the simple connectivity of the immediate basins of the super-attracting fixed points at $0$ and $\infty$.

The technical part of the paper is presented in Section \ref{sec:biaccessibility}, where we investigate the topology of periodic Fatou components associated with parameters in the Arnold tongues. Adapting the puzzle piece techniques of Yoccoz and Branner-Hubbard, we construct a sequence of nested compact sets formed by preimages of a fundamental domain (referred to as \emph{drops}, after \cite{Petersen2004}). Our first main result, Theorem~\ref{thm:biaccessibility}, establishes that, for adjacent parameters within the tongues, the periodic cycles in $\mathbb{S}^1\cap \mathcal{J}(B_a)$ are \textit{bi-accessible}, that is, accessible from the immediate basins of $0$ and $\infty$ (for a precise definition, consult Section~\ref{sec:biaccessibility}).

\begin{thmy}\label{thm:biaccessibility}
Let $(r, \alpha)$ be an adjacent parameter belonging to a tongue $T_{p/q}$, with $0<p/q<1$ in lowest terms. Then the map $B_{r,\alpha}$ has a bi-accessible cycle in its Julia set with rotation number $\rho=p/q$.
\end{thmy}

Finally, in Section \ref{sec:connectedness}, we connect these topological findings to the global geometry of the Julia set. The existence of bi-accessible points serves as an obstruction to the existence of multiply  connected basins. Combining our classification of Fatou components with the bi-accessibility result in Theorem~\ref{thm:biaccessibility}, we derive our second main result.

\begin{thmy}\label{thm:main-connectedness}
Let $d \in \mathbb{N}$ and consider the Blaschke family $B_{r,\alpha}$. For any parameters $r > 0$ and $\alpha \in [0,1)$, the Julia set $\mathcal{J}\left(B_{r,\alpha }\right)$ is connected if, and only if, the Fatou set $\mathcal{F}\left(B_{r,\alpha }\right)$ does not contain Herman rings.
\end{thmy}

Theorem~\ref{thm:main-connectedness} characterizes precisely when the Julia set of the generalized Blaschke product $B_{a}$ is connected, producing an explicit correspondence between bi-accessible points and the existence of rotation domains. A similar characterization was established by Bonifant, Buff, and Milnor \cite{Bonifant2018} for the family of antipode preserving cubic maps. In their work, the existence of bi-visible points also serves as the fundamental obstruction to the connectivity of the Julia set.

\section{Preliminaries}
\label{sec:preliminaries}

The construction of puzzle pieces relies on the contraction properties of inverse branches in hyperbolic regions. The following results from Steinmetz \cite{Steinmetz2011} and Beardon \cite{Beardon1991} are essential.

\begin{theorem}
	\label{thm:SteinmetzInverse}
	Let $D$ be a simply connected subdomain of $\widehat{\mathbb{C}}$ containing no critical values of $f$. Then there exist distinct analytic inverse branches of $f$ on $D$, defined by the preimages of a reference point.
\end{theorem}

\begin{lemma}
\label{lemma:diam0}
	Let $f$ be a rational map of degree $d\geq 2$. Suppose that $D$ is a simply connected domain intersecting the Julia set but disjoint from the forward orbit of the critical points. Let $\phi$ be an analytic inverse branch of $f$ defined in $D$ such that $\phi(D) \subset D$. Then, for any compact subset $K \subset D$,
	\[
		\operatorname{diam}\left(\phi ^{n}(K)\right) \to 0 \mbox{ as } n\to\infty.
	\]
\end{lemma}

The next result is a combination of Lemmas 3.7.9 and 3.7.10 found in \cite{DeFaria2008}.

\begin{lemma}
	\label{thm:Contraction}
	Let $U$ be a periodic component (or a suitable puzzle piece) where an inverse branch $g = f^{-1}$ is defined. If $g: U \to U$ is a strict contraction with respect to the Poincaré metric, then $g$ has a unique fixed point in $U$, which is the limit of all forward orbits under $g$ in $U$.
\end{lemma}

\subsection{Dynamics of Degree-One Circle Maps}
\label{sec:circle_maps}

We review a few concepts for the study of degree-one circle
maps. A comprehensive treatment of the theory can be found in \cite{DeFaria2008} and \cite{Alseda2000}.

Identify $\mathbb{S}^1$ with $\mathbb{R}/\mathbb{Z}$ and let $f\colon \mathbb{S}^1 \to \mathbb{S}^1$ be a continuous degree-one map. The space of lifts $F\colon \mathbb{R} \to \mathbb{R}$ (where $F(x+1)=F(x)+1$) is denoted by $\mathcal{L}_1$.

\begin{definition}
    \label{def:rot-number}
    The \emph{rotation number} of a lift $F \in \mathcal{L}_1$ at a point $x \in \mathbb{R}$ is
    \[ \rho(x, F) := \limsup_{n\to\infty} \frac{F^n(x)-x}{n}. \]
    The rotation set of $f$ is $\rho(f) = \{ \rho(x, F) \pmod 1 : x \in \mathbb{R} \}$.
\end{definition}

If $f$ is a homeomorphism, the rotation set is a singleton. For general degree-one maps, the rotation set is a closed interval, see \cite{Ito1981}. To avoid confusion with Definition~\ref{def:rot-number}, from now on we refer to $\rho(f)$ as the \emph{rotation interval} of $f$ in the case when $\rho(f)$ is not a singleton.

For circle diffeomorphisms, we provide the following constraint on the local stability of periodic orbits. 

\begin{lemma}
\label{lemma:non-repelling-cycle-existence}
    Let $f\colon \mathbb{S}^1 \to \mathbb{S}^1$ be an orientation-preserving diffeomorphism. If $f$ has a periodic orbit, then at least one periodic orbit is non-repelling.
\end{lemma}

\begin{proof}
If the rotation number of $f$ is irrational, then $f$ has no periodic points, and the lemma holds vacuously. We can therefore assume the rotation number is rational, $\rho(f)=p/q$, which guarantees the existence of periodic orbits with a period $q$. The analysis of $q$-periodic points of $f$ can be reduced to analyzing the fixed points of the iterated map $g := f^q$. A periodic orbit of $f$ is non-repelling if and only if its points are non-repelling fixed points for $g$. We proceed by contradiction, assuming that all periodic orbits of $f$ are repelling. This is equivalent to assuming that for every fixed point $z$ of $g$, the multiplier satisfies $g'(z) > 1$.

Consider the function $\psi\colon \mathbb{S}^1 \to \mathbb{R}$ defined by $\psi(z) = g(z) - z$. This function is well-defined on the circle, and its zeros correspond precisely to the fixed points of $g$. At any such fixed point $z_0$, our assumption implies that the derivative is positive:
$$\psi'(z_0) = g'(z_0) - 1 > 0.$$
This means that at every zero, the function $\psi$ is strictly increasing. However, a continuous function on a circle cannot have only zeros where it is strictly increasing. If we view the circle as the interval $[0,1]$ with identified endpoints and the function crosses the axis from negative to positive, it must eventually cross back from positive to negative to ensure $\psi(0) = \psi(1)$. At such zero, $z_c$, the derivative would have to be non-positive, $\psi'(z_c) \le 0$, which contradicts our assumption. Therefore, the initial assumption must be false. There must exist at least one fixed point $z_0$ of $g$ where $\psi'(z_0) \le 0$, which implies $g'(z_0) \le 1$. This point belongs to a non-repelling periodic orbit of $f$.
\end{proof}

\subsection{Rotation Sets for the Multiplication Map}

One of the main steps in the proof of biaccessible cycles will rely on the theory of rotation sets associated to the multiplication map $m_n(t) = nt \pmod{\mathbb{Z}}$ for $n \ge 2$. This theory has been extensively developed in \cite{Zakeri2018} from where we borrow the following exposition. 

\begin{definition}[Rotation Set]
    A non-empty compact set $X \subset \mathbb{R}/\mathbb{Z}$ is a \emph{rotation set for $m_n$} if it is invariant ($m_n(X)=X$) and $m_n$ preserves the circular order on $X$. The \emph{rotation number} $\rho(X)$ is defined via the map's degree-one monotone extension to the circle.
\end{definition}

If $\rho(X) = p/q$ is rational (in lowest terms), the dynamics are rigid: every forward orbit in $X$ is eventually periodic with period $q$. A rotation set is \emph{minimal} if it contains no proper subset that is also a rotation set. Every minimal rational rotation set is strictly a cycle of period $q$.

To classify these cycles, we examine their spatial distribution relative to the fixed points of $m_n$, given by $u_i = \frac{i}{n-1} \pmod{\mathbb{Z}}$ for $i=0, \dots, n-1$.

\begin{definition}[Deployment Vector]
    Let $X$ be a minimal rotation set with $\rho(X)=p/q$. The \emph{deployment vector} $\delta(X) = (\delta_1, \dots, \delta_{n-1})$ lies in the $(n-2)$-simplex $\Delta^{n-2}$, where each component $\delta_i$ represents the fraction of points in the cycle $X$ falling between consecutive fixed points:
    \[
        \delta_i = \frac{1}{q} \# \left\{ t \in X \colon t \in [u_{i-1}, u_i) \right\}.
    \]
\end{definition}

\begin{remark}
    For instance, under $m_3$, the cycle $X = \{1/8, 3/8\}$ has rotation number $1/2$. The fixed points are $0$ and $1/2$. Since both points of $X$ lie in $[0, 1/2)$, the deployment vector is $\delta(X) = (1)$.
\end{remark}

The following result establishes that the rotation number and deployment vector completely characterize the cycle.

\begin{theorem}[Goldberg \cite{Zakeri2018}]
\label{thm:Goldberg}
    For every rational $p/q \in (0,1)$ and every vector $\delta \in \Delta^{n-2}$ satisfying the integrality condition $q\delta_i \in \mathbb{Z}$, there exists a unique cycle $X$ for $m_n$ with rotation number $\rho(X) = p/q$ and deployment vector $\delta(X) = \delta$.
\end{theorem}


\section{The Blaschke Family}
\label{sec:TheBlaschke}

This section investigates the dynamics of the Blaschke family $B_{c,t}$. We begin by defining the family and establishing a conformal conjugacy that reduces the study to a single complex parameter $a$. Then, we provide an analysis of the critical orbits in the dynamical plane and define the cicle maps associated to the Blaschke product restricted to $\mathbb{S}^1$. We conclude the section with a description of the topology and connectivity of Fatou components, including the proof for the absence of Siegel disks in this family.

While classical finite Blaschke products restrict their parameters to the unit disk $\mathbb{D}$ to map the unit circle $\mathbb{S}^1$ to itself, we study a broader family of rational maps. For $d \in \mathbb{N}$, $t \in [0,1)$, and any parameter $c \in \mathbb{C}$, we define the \emph{generalized Blaschke products} as:
\begin{equation}
    \label{eq:Blaschke}
    B_{c,t}(z) = e^{2\pi i t} z^{d+1} \left(\dfrac{z - c}{1 - \bar{c}z}\right)^d.
\end{equation}
Even when $a \in \mathbb{C}$, these maps preserve the essential Blaschke property of leaving $\mathbb{S}^1$ invariant. As a result, the dynamics on the Riemann sphere are symmetric with respect to the unit circle, in the sense that $B_{c,t}(z) = \mathcal{I} \circ B_{c,t} \circ \mathcal{I}(z)$, where $\mathcal{I}(z) = 1/\bar{z}$.

To reduce the dimension of the parameter space, we utilize the following conjugacy result.

\begin{lemma}
    \label{lemma:conjugacy-t-to-0}
    Let $\alpha \in \mathbb{R}$ and define the rotation $\eta(z) = e^{-2\pi i \alpha} z$. Then $\eta$ conjugates $B_{c,t}$ to $B_{b,s}$, where $s = t + 2d\alpha \pmod 1$ and $b = \eta(c)$. Specifically, by choosing $\alpha = -t/(2d)$, the map $B_{c,t}$ is conformally conjugate to $B_{a, 0}$ with $a = c e^{i \pi t/d}$.
\end{lemma}

This lemma ensures that the dynamical features of the full family are completely captured by the case $t=0$ (up to a rotation of the parameter $a$). Therefore, without loss of generality, we restrict our analysis to the one-parameter family $B_a := B_{a,0}$, defined by:
\begin{equation}
    \label{eq:Blaschke-sin-t}
    B_a(z) = z^{d+1} \left( \frac{z - a}{1 - \bar{a} z} \right)^d.
\end{equation}

The reduction of the parameter space is illustrated in Figure \ref{fig:Conjugation}.

\begin{figure}[!ht]
    \centering
    \includegraphics[width=0.8\linewidth]{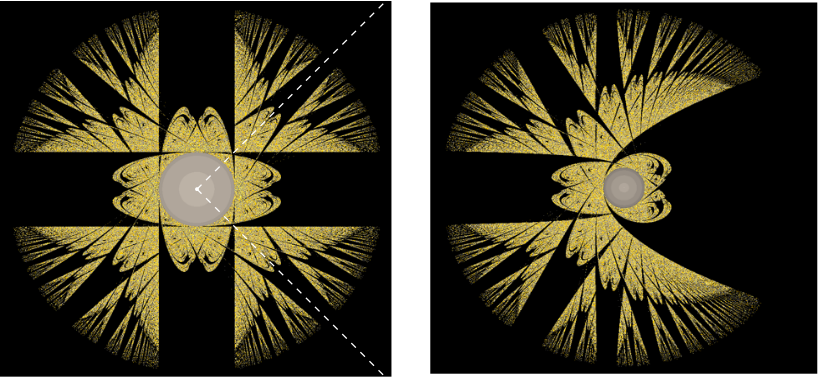}
    \caption{To the left, partial representation of the parameter space of $B_{a,0}$ for $d=2$. To the right, partial representation of the reduced parameter space of $B_{a,0}$ for $d=2$.}
    \label{fig:Conjugation}
\end{figure}

\subsection{Critical Orbits}
\label{sec:critical-orbits}

The map $B_a$ has degree $2d+1$, possessing $4d$ critical points counted with multiplicity. Differentiating \eqref{eq:Blaschke-sin-t}, we obtain:
\begin{equation}
    \label{eq:Blaschke-Derivative}
    B'_{a}(z)= - \dfrac{z^d\left(z-a\right)^{d-1}}{\left(1-\bar{a}z\right)^{d+1}}\cdot h(z),
\end{equation}
where $h(z)=\bar{a}(d+1)z^2 - (2d+1+|a|^2)z + a(d+1)$. The critical set consists of:
\begin{itemize}
    \item The super-attracting fixed points $0$ and $\infty$ (multiplicity $d$);
    \item The points $a$ and $1/\bar{a}$ (multiplicity $d-1$ if $d \ge 2$);
    \item Two \emph{free critical points} $c_\pm$ (roots of $h(z)$), given by:
\end{itemize}
\begin{equation}
    \label{eq:both-critical-points}
    c_\pm = a \cdot \dfrac{2d+1+|a|^2 \pm \sqrt{\Delta_a}}{2(d+1)|a|^2}, \quad \text{where } \Delta_a = \left(|a|^2-(2d+1)^2\right)\left(|a|^2-1\right).
\end{equation}

The positions of $c_\pm$ relative to the unit circle $\mathbb{S}^1$ and the parameters determine the dynamical regime. We summarize their properties below (see also Figure \ref{fig:position-critical-points}).

\begin{figure}
\centering
     \includegraphics[width=0.8\linewidth]{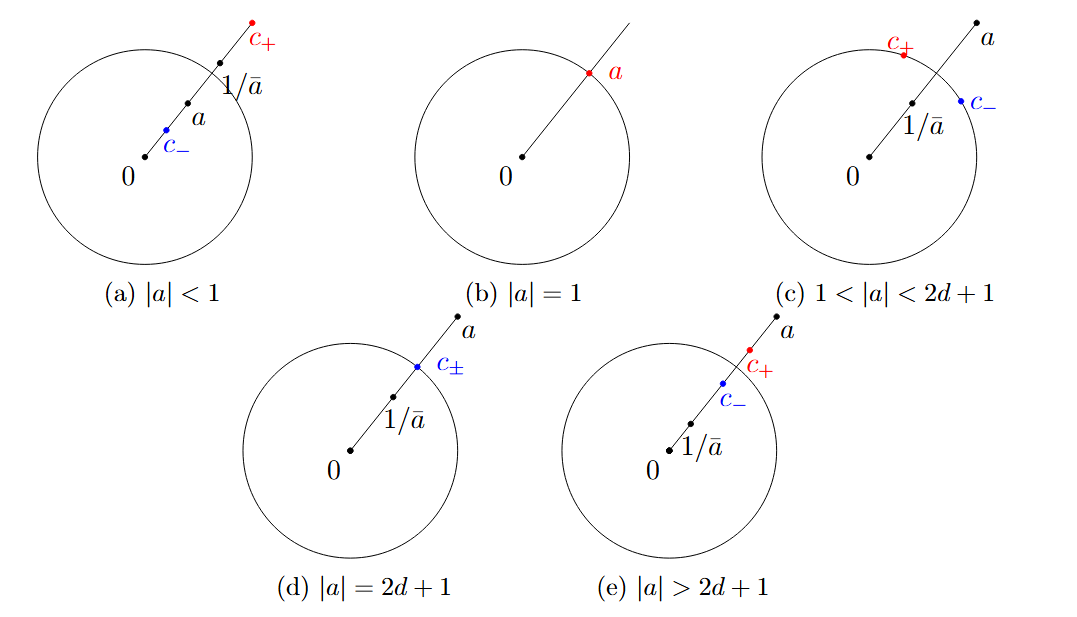}
    \caption{Location of the free critical points $c_\pm$ relative to the unit circle and the parameters $a, 1/\bar{a}$.}
    \label{fig:position-critical-points}
\end{figure}

\begin{proposition}[Location of Free Critical Points]
\label{prop:critical-points-location}
    The location of $c_\pm$ depends on $|a|$ as follows:
    \begin{enumerate}
        \item \textbf{$|a| > 2d+1$:} The points $c_\pm$ are symmetric with respect to $\mathbb{S}^1$ ($c_+ = 1/\bar{c}_-$). They lie on the ray passing through $a$ such that $1 < |c_+| < |a|$ and $|1/\bar{a}| < |c_-| < 1$.
        \item \textbf{$|a| < 1$:} Symmetry holds ($c_+ = 1/\bar{c}_-$). They lie on the ray passing through $a$ such that $|1/\bar{a}| < |c_+|$ and $|c_-| < |a| < 1$.
        \item \textbf{$1 \leq|a| \leq2d+1$:} The points $c_\pm$ lie on the unit circle $\mathbb{S}^1$. They are distinct if $1 < |a| < 2d+1$ and coincide if $|a|=1$ or $|a|=2d+1$.
    \end{enumerate}
\end{proposition}

\begin{proof}
    \textit{Cases 1 \& 2 ($|a| > 2d+1$ or $|a| < 1$):} In these regions, the discriminant $\Delta_a$ in \eqref{eq:both-critical-points} is positive, implying $c_\pm$ lie on the same ray as $a$. The symmetry $c_+ = 1/\bar{c}_-$ follows from the conjugate-reciprocal coefficients of $h(z)$.
    For $|a| > 2d+1$, the condition $|c_+| < |a|$ is equivalent to $\sqrt{\Delta_a} < (2d+1)(|a|^2-1)$. Squaring both sides (valid as terms are positive) yields the true inequality $|a|^2(1-(2d+1)^2) < 0$. Since $c_+ \in \mathbb{D}^c$, we have $1 < |c_+| < |a|$. The position of $c_-$ follows by inversion.
    For $|a| < 1$, similar algebraic manipulation shows $|c_-| < |a|$.

    \textit{Case 3 ($1 \leq|a| \leq2d+1$):} Here $\Delta_a \leq0$, making the square root purely imaginary. Thus $|c_\pm|^2 = c_+ \bar{c}_+ = 1$, placing them on $\mathbb{S}^1$.
\end{proof}

Based on Proposition~\ref{prop:critical-points-location}, we classify the parameter space into the following three regions.

\paragraph{Region I: Trivial Dynamics ($|a| \leq 1$)}
If $|a| < 1$, the pole $1/\bar{a}$ lies outside $\overline{\mathbb{D}}$. Symmetry implies $\mathcal{A}^*(0) = \mathbb{D}$ and $\mathcal{A}^*(\infty) = \widehat{\mathbb{C}} \setminus \overline{\mathbb{D}}$.
If $|a| = 1$, the map simplifies to the monomial $B_a(z) = (-a)^d z^{d+1}$ for $z \neq a$.
In both cases, we have:
\begin{lemma}
\label{lemma:Julia-is-S1-a-less-1}
    If $|a| \leq 1$, the Julia set is the unit circle, $\mathcal{J}(B_a) = \mathbb{S}^1$.
\end{lemma}

\paragraph{Region II: Endomorphism Dynamics ($1 < |a| < 2d+1$).}
This region contains the most complex boundary dynamics, as the free critical points $c_\pm$ are distinct and lie on $\mathbb{S}^1$. Consequently, the restriction $B_a|_{\mathbb{S}^1}$ is a non-injective degree-one circle endomorphism. The rotation set for these maps is a non-empty closed interval rather than a single number.

\paragraph{Region III: Homeomorphism Dynamics ($|a| \geq 2d+1$).}
Here, at most one critical point lies on the circle.
\begin{itemize}
    \item If $|a| > 2d+1$, $c_\pm$ are symmetric but off the circle. The map $B_a|_{\mathbb{S}^1}$ is an analytic circle diffeomorphism.
    \item If $|a| = 2d+1$, $c_\pm$ collide into a single critical point on $\mathbb{S}^1$. The map $B_a|_{\mathbb{S}^1}$ is an orientation-preserving homeomorphism.
\end{itemize}
In this region, the map possesses a unique, well-defined rotation number.

\subsection{Blaschke Circle Maps and Arnold Tongues}
\label{section:reparametrization}

To apply the theory of rotation intervals \cite{Boyland1986, Canela2016}, we reparametrize the Blaschke family on $\mathbb{S}^1$ as a family of degree-one circle maps. We restrict the parameter space to $r \in (1, \infty)$ and the fundamental domain $\alpha \in (-\frac{1}{4d}, \frac{1}{4d}]$. The resulting map $g_{r,\alpha}$ and its lift $G_{r,\alpha}$ (chosen via the principal branch of the logarithm) are defined as:
\begin{align}
    \label{eq:blaschke-circle-maps}
    g_{r,\alpha}(e^{2\pi ix}) &= e^{2\pi i 2d\alpha }e^{2\pi i (d+1)x}\left(\dfrac{e^{2\pi ix}-r}{1-re^{2\pi ix}}\right)^d, \\
    \label{eq:lift-logarithm}
    G_{r,\alpha}(x) &= (d+1)x+2d\alpha +\dfrac{d}{2\pi i}\log \left(\dfrac{e^{2\pi ix}-r}{1-re^{2\pi ix} }\right).
\end{align}
The dynamical nature of $g_{r,\alpha}$ is determined by $r$ as follows (compare with Proposition~\ref{prop:critical-points-location}).

\begin{proposition}
    \label{prop:remarks}
    The map $g_{r,\alpha}$ is a degree-one circle map, whose behavior depends on $r$ in the following manner.
    \begin{enumerate}
        \item[a)] For $r > 2d+1$, it is an orientation-preserving diffeomorphism.
        \item[b)] For $r = 2d+1$, it is an orientation-preserving homeomorphism.
        \item[c)] For $1 < r < 2d+1$, it is a continuous endomorphism.
    \end{enumerate}
\end{proposition}

\begin{proof}
    Let $z\in \mathbb{S}^1$. Then the map can be factored as $g_{r,\alpha}(z) = e^{2\pi i (2d\alpha)} \cdot z^{d+1} \cdot (\phi_r(z))^d$, where $\phi_r(z) = \frac{z-r}{1-rz}$. Since $\phi_r$ is an orientation-reversing homeomorphism (hence of degree $-1$), its principal lift $\hat{\phi}_r$ satisfies $\hat{\phi}_r(x+1) = \hat{\phi}_r(x)-1$. The lift $G_{r,\alpha}(x) = (d+1)x + 2d\alpha + d\hat{\phi}_r(x)$ therefore satisfies:
    \[
        G_{r,\alpha}(x+1) = (d+1)(x+1) + 2d\alpha + d(\hat{\phi}_r(x)-1) = G_{r,\alpha}(x) + 1.
    \]
    Thus, $g_{r,\alpha}$ has degree one. To classify the morphism type of $g_{r,\alpha}$, we analyze the derivative
    \begin{equation}
        \label{eq:lift-derivative}
        G'_{r,\alpha}(x) = d+1 + d \cdot \frac{1-r^2}{1-2r\cos(2\pi x)+r^2}.
    \end{equation}
    Observe that $g_{r,\alpha}$ fails to be a homeomorphism if and only if critical points lie on $\mathbb{S}^1$, which occurs when $1 \leq r \leq2d+1$. Therefore,
    \begin{itemize}
        \item[a)] if $r > 2d+1$, the critical points are strictly outside $\mathbb{S}^1$, implying $G'_{r,\alpha}(x) > 0$ everywhere, so $g_{r,\alpha}$  is a diffeomorphism.
        \item[b)] If $r = 2d+1$, then critical points of collide on $\mathbb{S}^1$. This implies $G'_{r,\alpha}(x) \ge 0$, vanishing only at the lifts of the critical point. Consequently $g_{r,\alpha}$ remains injective, hence a homeomorphism.
        \item[c)] If $1 < r < 2d+1$, two distinct critical points lie on $\mathbb{S}^1$, so $G'_{r,\alpha}(x)$ vanishes at the lifts of those two points. The map $g_{r,\alpha}$ is a non-injective endomorphism. \qedhere
    \end{itemize}
\end{proof}

In summary, a generalized Blaschke product $B_{c,t}$ is conformally conjugated to a Blaschke product $B_a=B_{a,0}$ depending on a single complex parameter $a\in \mathbb{C}$. For those parameters with norm $|a|\geq 1$, the restriction of $B_a|\mathbb{S}^1$ defines a circle map $g_{r,\alpha}$ with $r=|a|$ and $\arg(a)=2d\alpha$.

We now present the definition of an \emph{Arnold tongue} for the family $g_{r,\alpha}$, which is a
 useful concept introduced by Arnol'd for the study of bifurcation spaces of 2-parameter circle maps (see \cite{Boyland1986}): We define the \emph{Arnold tongue} of rotation number $s \in \mathbb{R}$ as the set of parameters
 \[T_s = \{ (r,\alpha) \colon s \in \rho(g_{r,\alpha}) \},\]
 where $\rho$ denotes the rotation interval of $g_{r,\alpha}$. 

\noindent We are particularly interested in tongues for rational rotation number $s=p/q$, where $(p,q)=1$. Within the tongues $T_{p/q}$, we identify the \emph{adjacent region} $AT_{p/q}$ as the set of parameters $(r,\alpha)$ with $r \in (1, 2d+1)$ where both free critical points of $g_{r,\alpha}$ lie in the immediate basin of an attracting (or superattracting) fixed point $z_0 \in \mathbb{S}^1$. In Section~\ref{sec:biaccessibility} the adjacent parameters will become of utmost interest in the proof of Theorem~\ref{thm:biaccessibility}.

\subsection{Fatou Components of the Blaschke family}

This section classifies the Fatou components of the Blaschke family $B_{a}$. The existence and topology of these components are strictly determined by the location of the critical points and the parameter $a$.

We begin by ruling out the existence of Siegel disks for the entire family and characterizing the existence of Herman rings.

\begin{theorem}[Rotation Domains]
\label{thm:rotation-domains}
    Let $B_a$ be the Blaschke product defined in \eqref{eq:Blaschke-sin-t}.
    \begin{enumerate}
        \item The map $B_a$ has no Siegel disks for any $a \in \mathbb{C}$.
        \item If $|a| \leq 2d+1$, the map $B_a$ has no Herman rings.
    \end{enumerate}
\end{theorem}

\begin{proof}
    \textit{1. Absence of Siegel Disks:} We proceed by contradiction. Assume $B_a$ possesses a Siegel disk $\Delta$ of period $p \ge 1$. Since $\mathbb{S}^1$ is totally invariant and contains super-attracting points in its complement, $\Delta$ must be disjoint from $\mathbb{S}^1$. By the symmetry $\mathcal{I}(z)=1/\bar{z}$, there exists a disjoint mirror Siegel disk $\widetilde{\Delta} = \mathcal{I}(\Delta)$. These two disjoint cycles must attract distinct critical orbits. Since the orbits of $0, \infty, a, 1/\bar{a}$ are captured by the super-attracting fixed points, only the free critical points $c_+$ and $c_-$ remain. Thus, one attracts to $\Delta$ and the other to $\widetilde{\Delta}$.
    This configuration yields four non-repelling cycles: $\{0\}, \{\infty\}, \Delta, \widetilde{\Delta}$. This saturates the Shishikura inequality ($4 \leq 4$). However, $B_a|_{\mathbb{S}^1}$ is a circle diffeomorphism and must possess a non-repelling cycle on the circle. This fifth cycle contradicts the inequality. Thus, no Siegel disks exist.

    \textit{2. Absence of Herman Rings:} If a rational function has a Herman ring, two critical orbits must accumulate on different boundary components \cite{Shishikura1987}. For $1 \leq |a| \leq 2d+1$, both free critical points $c_\pm$ lie on the invariant circle $\mathbb{S}^1$, preventing them from accumulating on distinct boundaries. If $|a|<1$, the Julia set is the circle itself, precluding the existence of Fatou components other than the basins of $0$ and $\infty$.
\end{proof}

\begin{remark}
    In contrast to Theorem \ref{thm:rotation-domains}(2), if $|a| > 2d+1$, Herman rings may exist. In this range, the rotation number $\rho(t)$ of $B_{a,t}|_{\mathbb{S}^1}$ depends linearly on $t$ (specifically $\rho(t) = \rho_0 + t \pmod 1$). By choosing $t$ such that $\rho(t)$ is a Brjuno number, one can generate Herman rings as described in  \cite{Chu2018}.
\end{remark}

Next, we describe the topology of the immediate basins of the super-attracting fixed points $0$ and $\infty$.

\begin{theorem}[Basins of $0$ and $\infty$]
\label{thm:super-attracting-basins}
    For $|a| > 1$, the free critical points $c_\pm$ do not lie in the immediate basins $\mathcal{A}^*(0)$ or $\mathcal{A}^*(\infty)$. Consequently:
    \begin{enumerate}
        \item $\mathcal{A}^*(0)$ and $\mathcal{A}^*(\infty)$ are simply connected.
        \item The boundary $\partial\mathcal{A}^*(\infty)$ contains exactly $d$ fixed points (and by symmetry, so does $\partial\mathcal{A}^*(0)$).
    \end{enumerate}
\end{theorem}

\begin{proof}
    First, we locate the critical points. For $1<|a|\leq 2d+1$, $c_\pm \in \mathbb{S}^1$. Since $\mathbb{S}^1$ is invariant, $c_+ \notin \mathcal{A}^*(\infty)$. For $|a|>2d+1$, assume for contradiction that $c_+ \in \mathcal{A}^*(\infty)$. Let $U_0 \subset \mathcal{A}^*(\infty)$ be the maximal Böttcher domain. The invariant circle $\mathbb{S}^1$ cannot be contained in $U_0$ because the map is $1$-to-$1$ on the circle but $(d+1)$-to-$1$ in the basin. However, the preimage of the image of $U_0$ must contain the pole $1/\bar{a} \in \mathbb{D}$. This implies $\mathcal{A}^*(\infty)$ connects $\infty$ to $\mathbb{D}$, crossing $\mathbb{S}^1$. A basin intersecting the invariant circle must be symmetric, implying it contains both $0$ and $\infty$, a contradiction. Thus $c_+ \notin \mathcal{A}^*(\infty)$ and by symmetry $c_- \notin \mathcal{A}^*(0)$.

    \textit{Connectivity:} Since $\mathcal{A}^*(\infty)$ contains no free critical points, the map $B_a: \mathcal{A}^*(\infty) \to \mathcal{A}^*(\infty)$ is proper of degree $d+1$. By the Riemann-Hurwitz formula, if $m$ is the connectivity, $m-2 = (d+1)(m-2) + d$, implying $m=1$. Thus, the basin is simply connected. (Preimages are handled similarly via induction; components disjoint from $\mathbb{S}^1$ contain at most one critical point or the pole, preserving simple connectivity).

    \textit{Boundary Fixed Points:} Since $\mathcal{A}^*(\infty)$ contains only one critical point ($\infty$), the Böttcher coordinate $\phi$ conjugates $B_a$ to $z \mapsto z^{d+1}$ on $\mathbb{D}$. The $d$ fixed points of $z^{d+1}$ on $\mathbb{S}^1$ correspond via the homeomorphism $\phi^{-1}$ to $d$ fixed points on $\partial\mathcal{A}^*(\infty)$.
\end{proof}

Finally, we classify all other periodic Fatou components.

\begin{theorem}[General Periodic Components]
\label{thm:general-components}
    Let $|a| \ge 2d+1$. Let $U$ be a periodic Fatou component that is not a Herman ring and not $\mathcal{A}^*(0)$ or $\mathcal{A}^*(\infty)$. Then:
    \begin{enumerate}
        \item $U$ is symmetric under involution, i.e., $\mathcal{I}(U) = U$.
        \item The closure $\overline{U}$ intersects $\mathbb{S}^1$.
        \item $U$ is simply connected.
    \end{enumerate}
\end{theorem}

\begin{proof}
    \textit{1. Symmetry and Intersection:} First, we note that $B_a$ has no super-attracting cycles other than $0$ and $\infty$ for $|a|>2d+1$ (similar proof to Theorem \ref{thm:rotation-domains}, counting critical orbits).
    Assume $U$ is a fixed component (the periodic case follows by considering $B_a^k$). If $\overline{U} \cap \mathbb{S}^1 = \emptyset$, then by symmetry, there exists a distinct component $V = \mathcal{I}(U)$. The set of non-repelling cycles would then include $\{0\}, \{\infty\}, U, V$. This saturates the Shishikura bound ($4 \leq 4$). However, $B_a|_{\mathbb{S}^1}$ must admit a non-repelling cycle, creating a 5th cycle, which is a contradiction. Thus, $\overline{U} \cap \mathbb{S}^1 \neq \emptyset$.
    Since $U$ intersects the invariant set $\mathbb{S}^1$, its associated limit set (e.g., attracting point $p$) lies on $\mathbb{S}^1$. Since $\mathcal{I}(p)=p$, standard convergence arguments imply $z \in U \iff \mathcal{I}(z) \in U$, so $\mathcal{I}(U)=U$.

    \textit{2. Connectivity:} Because $U$ is symmetric and intersects $\mathbb{S}^1$, its immediate basin must contain both free critical points (due to symmetry) or the single critical point if they coincide.
    If $U$ is an immediate attracting or parabolic basin, let $U_0 \subset U$ be the maximal domain for linearized coordinates. $\partial U_0$ contains both $c_+$ and $c_-$. The preimage $U_1 = B_a^{-1}(U_0)$ containing the fixed point maps properly with degree 3. By Riemann-Hurwitz, $m_{U_1}-2 = 3(1-2)+2 \implies m_{U_1}=1$. The sequence of preimages forms a nested exhaustion of simply connected sets, so $U$ is simply connected.
\end{proof}

\section{Bi-accessible Cycles}
\label{sec:biaccessibility}

This section analyzes the dynamics of the Blaschke family 
\[
B_a(z) = z^{d+1} \left( \frac{z - a}{1 - \bar{a} z} \right)^d,
\]
focusing on adjacent parameters $a=re^{2\pi i \alpha}$ with $r>1$. Recall that the restriction of this map to the unit circle may act as a homeomorphism, endomorphism, or diffeomorphism, depending on the value of $r$.

\begin{definition}
 \label{def:bi-accessible}
    A point $z_0\in \mathcal{J}(B_{r,\alpha})$ is \emph{bi-accessible} if it is the landing point of a Böttcher ray from $\mathcal{A}^*(\infty)$ and a Böttcher ray from $\mathcal{A}^*(0)$ simultaneously. Due to the symmetry $\mathcal{I}(z)=1/\bar{z}$, any bi-accessible point must belong to the unit circle $\mathbb{S}^1$.
\end{definition}

The main goal of this section is to prove the existence of bi-accessible cycles for these parameters, as described in Theorem~\ref{thm:biaccessibility}. This topological property is a prerequisite for the connectivity arguments presented in Section~\ref{sec:connectedness}. We provide the statement of theorem again for the reader's convenience.

\vspace{0.5cm}

\noindent
{\small {\bf Theorem A} \textit{Let $(r,\alpha)$ be an adjacent parameter that belongs to a tongue $T_{p/q}$, with $0<p/q<1$ in lowest terms. Then the map $B_{r,\alpha}$ has a bi-accessible cycle in its Julia set $\mathcal{J}(B_{r,\alpha})$ with rotation number $\rho = p/q$.}}

\vspace{0.5cm}

The strategy of our proof is constructive, relying on the puzzle pieces technique developed by Yoccoz and Branner-Hubbard for polynomials, then extended for rational functions in \cite{Petersen2004}. We note that the specific use of bi-accessibility as a topological obstruction to connectivity was inspired by the work of Bonifant et al. \cite{Bonifant2018}, who employed a similar strategy to analyze the topology of Fatou components for cubic maps.

We construct a nested sequence of compact sets in the dynamical plane that isolate the cycle of interest for the specific case of the tongue $T_{1/2}$ and any given value of $d\geq 1$. We handle the homeomorphism, endomorphism, and diffeomorphism cases separately. The proof for any rational tongue $T_{p/q}$ follows from the interval partition established in Proposition~\ref{prop:gen_interval}.

\begin{proposition}
    \label{prop:boundary_dynamics_and_fixed_points}
    Let $B = B_{r,\alpha}$ be a Blaschke map of degree $2d+1$ with adjacent parameters. The restriction $B|_{\partial\mathcal{A}^*(\infty)}$ is topologically conjugate to the multiplication map $m_{d+1}(t) = (d+1)t \pmod{\mathbb{Z}}$. Consequently, $B$ possesses exactly $2d+2$ distinct fixed points in $\widehat{\mathbb{C}}$, distributed as follows:
    \begin{enumerate}
        \item The super-attracting fixed points $0$ and $\infty$;
        \item $d$ points on the boundary $\partial\mathcal{A}^*(\infty)$;
        \item $d$ points on the boundary $\partial\mathcal{A}^*(0)$.
    \end{enumerate}
\end{proposition}

\begin{proof}
    First, we establish the conjugacy. As $\mathcal{A}^*(\infty)$ is simply connected and the map is hyperbolic for adjacent parameters, the Böttcher coordinate $\phi\colon \mathcal{A}^*(\infty) \to \mathbb{D}$ extends to a homeomorphism $\phi\colon \overline{\mathcal{A}^*(\infty)} \to \overline{\mathbb{D}}$. This map conjugates $B$ on the closure of the basin to the monomial map $w \mapsto w^{d+1}$ on $\overline{\mathbb{D}}$. Restricting this action to the boundary $\mathbb{S}^1$, the dynamics correspond to the angle doubling map $m_{d+1}(t)$.

    Next, we classify the fixed points. A rational map of degree $2d+1$ has $2d+2$ fixed points in $\widehat{\mathbb{C}}$ (counted with multiplicity). We explicitly identify them:
    \begin{itemize}
        \item The points $0$ and $\infty$ are super-attracting fixed points by construction.
        \item The map $m_{d+1}$ on the circle has exactly $d$ distinct fixed points given by angles $t = k/d$ for $k=0, \dots, d-1$. via the conjugacy $\phi$, these correspond to $d$ distinct fixed points on $\partial\mathcal{A}^*(\infty)$.
        \item By the symmetry of the Blaschke product, the map is invariant under the involution $\mathcal{I}(z) = 1/\bar{z}$. Therefore, the $d$ fixed points on $\partial\mathcal{A}^*(\infty)$ map to $d$ distinct fixed points on $\partial\mathcal{A}^*(0)$.
    \end{itemize}
    Summing these contributions yields $2 + d + d = 2d+2$ distinct fixed points. Thus, all fixed points of $B$ are simple and accounted for.
\end{proof}

\subsection{Admissible Words}
\label{sec:rules-admissibility}

To facilitate subsequent constructions within the three possible scenarios where $B|_{\mathbb{S}^1}$ acts as a diffeomorphism, homeomorphism, or endomorphism, we establish a criterion for forming words that will serve as labels for certain sets (the \emph{drops}). For a general degree $d \ge 1$, let $\mathcal{A}_d$ denote an alphabet consisting of $2d+2$ symbols: $\underline{0},\underline{1},\dots,\underline{d},0,1,\dots,d$. For $k\geq 1$, a word of length $k$ given by $i_1\dots i_k$, with $i_j\in\mathcal{A}_d$, is said to be an \emph{admissible word} if it satisfies the following rules:

\begin{itemize}
	\item[$a)$] The last symbol is not underlined: $i_k\not\in \left\{\underline{0},\underline{1},\dots,\underline{d}\right\}$.
	\item[$b)$] An underlined symbol must be followed by $d$: $i_j\in \left\{\underline{0},\underline{1},\dots,\underline{d}\right\} \Rightarrow i_{j+1}=d$.
	\item[$c)$] A non-underlined symbol cannot be followed by $d$: $i_j\in\left\{0,1,\dots,d\right\} \Rightarrow i_{j+1}\neq d$.
\end{itemize}

We denote the set of finite admissible words with symbols in $\mathcal{A}_d$ as $\Xi_{\mathcal{A}_d}$.

\begin{remark}
It is important to note that for any $d\geq 1$, there are cases where a word $i_1\dots i_k$ is admissible but its prefix $i_1\dots i_{k-1}$ is not. This happens if and only if $i_{k-1}\in\left\{\underline{0}, \dots, \underline{d}\right\}$. However, by the admissible rules, if $i_1\dots i_k$ is admissible, then the prefix $i_1\dots i_{k-2}$ must also be admissible since $i_{k-2}$ cannot be an underlined symbol $\underline{0},\underline{1},\dots,\underline{d}$.
\end{remark}

\subsection{The construction}
We will examine parameters $(r,\alpha )\in T_{1/2}$ for which $(r,\alpha)$ is an adjacent parameter. It is important to emphasize that there are subtle differences in the upcoming constructions of the puzzle pieces depending on whether  $B|_{\mathbb{S}^1}$ acts as a homeomorphism, an endomorphism or a diffeomorphism, and the cycle that traps the two free critical points is super-attracting or attracting. Therefore, we will treat them as distinct cases.

Due to the symmetry of $B$ with respect to $\mathbb{S}^1$, the preimage of $\mathbb{S}^1$ under $B$ encloses two distinct Jordan domains. One of these domains is located within $\mathbb{D}$, while the other resides in $\widehat{\mathbb{C}}\setminus \overline{\mathbb{D}}$.

We will denote the bounded domain within $\widehat{\mathbb{C}}\setminus \overline{\mathbb{D}}$ as $U$ and establish the positive orientation of its boundary with respect to the zero of $B$ at $a\in\mbox{int}\left(\overline{U}\right)$.  

Since $z=a\in U$ is a critical point of multiplicity $d-1$, then
\[
	B|_{\overline{U}}\colon \overline{U}\xrightarrow[]{\ d:1\ } \overline{\mathbb{D}}
\]	 
is a branched covering map of degree $d$, with  $a$ as a unique branched point.

\begin{definition}[Drops and roots]
\label{def:drops-roots-homeo}
For $k\geq 0$, we define a \emph{drop of level $k$} as any connected component of $B^{-k}\left(U\right)$, where $U$ itself is defined as a \emph{drop of level $0$}. Since $U$ is a Jordan domain, any drop of level $k$ is also a Jordan domain. If $V$ is a drop of level $k$, we define the \emph{root of $V$} as the unique boundary point given by $\{z\}=B^{-k}(c)\cap \partial V$.
\end{definition}

\subsubsection{The Homeomorphism Case}
\label{subsec:HomeoCase}

In this case, the Blaschke product $B$ has a single critical point of multiplicity $2$ on the unit circle, denoted as $c=e^{2\pi i\alpha }$. The preimage of $\mathbb{S}^1$ under $B$ takes the shape of a figure-eight curve at $c$ (see Figure \ref{fig:Step02}).

\paragraph{Super-attracting case}
We focus on the following specific setting: $d \ge 2$, $r = 2d + 1$, and $\alpha$ for which the critical point $c$ belongs to a super-attracting $2$-cycle on $\mathbb{S}^1$ with rotation number $1/2$. We denote this cycle by $\{c, v\}$. Observe that $v = B(c)$ is a critical value that also lies on $\mathbb{S}^1$. Because the closure $\overline{U}$ intersects $\mathbb{S}^1$ only at the critical point $c$ (its root), and since $v \neq c$, we conclude that $v$ does not belong to $\overline{U}$. Consequently, the only point in the forward orbit of $c$ that lies in $\overline{U}$ is $c$ itself, i.e. $\overline{U} \cap \mathcal{O}^+(c) = \{c\}$.

\vspace{0.5cm}
\noindent\textbf{Labeling drops and roots.}
 We will assign admissible words as labels of drops at each level, following the admissibility rules given in Subsection \ref{sec:rules-admissibility}.

\vspace{0.5cm}
\noindent\textbf{Drops of level $1$}

For $i = 0,\dots , d-1$, we denote as $U_i$ the connected component of $B^{-1}(U)$ such that $\overline{U}_i \cap \overline{U}\neq \emptyset$. Each such drop $U_i$ has its root at the point $z_i \in B^{-1}(c) \cap \partial U_i$. The indices $i$ are assigned in increasing manner by following the positive orientation of $\partial U$ starting at its root at $c$. 

Additionally, we denote by $U_d$ the connected component of $B^{-1}(U)$ such that $B^{-1}(U)\cap\mathbb{S}^1\neq\emptyset$. It is worth noting that $U_d$ is the unique drop of level $1$ with its root $z_d$, which coincide with the critical value $v$, located on the unit circle.

For each $i = 0,1,\dots , d$, the map 
\[
	B|_{U_i }\colon \overline{U}_i\to \overline{U}
\]
is holomorphic, bijective, and $B'(z)|_{U_i}\neq 0$ for all $z\in \overline{U}_i $, given that the critical point $c$ and its positive orbit belong to $\mathbb{S}^1$. Consequently, $B|_{\overline{U}_i}$ is a biholomorphism.

\vspace{0.5cm}
\noindent\textbf{Drops of level $2$}

As before, we start by labeling the drops of level two that do not intersect the unit circle. In this context, let $i=0,1,\dots , d$ and $j=0,\dots , d-1$. We define $U_{ij}$ as the connected component of $B^{-1}\left(U_{j}\right)$, such that $\overline{U}_{ij} \cap \overline{U}_{i}\neq \emptyset $. As previously, this labeling is increasing with respect to the positive orientation along $\partial U_{i}$, starting at its root.

As $\partial U_d$ contains a critical value, $z_d$, we state the following lemma, whose proof relies on analyzing the structure of preimages of boundary curves near critical points, similar to the arguments used in Proposition 3.1 of \cite{Canela2020} .

\begin{lemma}
\label{lemma:lobes}
There exists a preimage of $\partial U_d$ under $B$ within $\mathbb{C}\setminus \mathbb{D}$, forming a figure-eight curve. Additionally, each of the lobes of this figure-eight curve lies on opposite sides of the straight line connecting $0$, $c$, and $a$.
\end{lemma} 

\begin{proof}
As $\overline{U}_d$ is a Jordan domain, its boundary $\partial U_d$ forms a simple closed curve that contains the critical value $z_d$. Let $\gamma$ be the connected component of $B^{-1}(\partial U_d)$ in $\mathbb{C}\setminus\mathbb{D}$ that contains $c$. Thus, $\gamma$ consists of two simple closed curves, denoted as $\gamma_1$ and $\gamma_2$, intersecting at the point $c$.

Each $\gamma_1$ and $\gamma_2$ define the boundaries of two distinct connected components on $\mathbb{C}$. In general, we have two different arrangements for $\gamma_1$ and $\gamma_2$, as shown in Figure \ref{Figure:ocho}.

\begin{figure}
\begin{center}
	\includegraphics[scale=0.5]{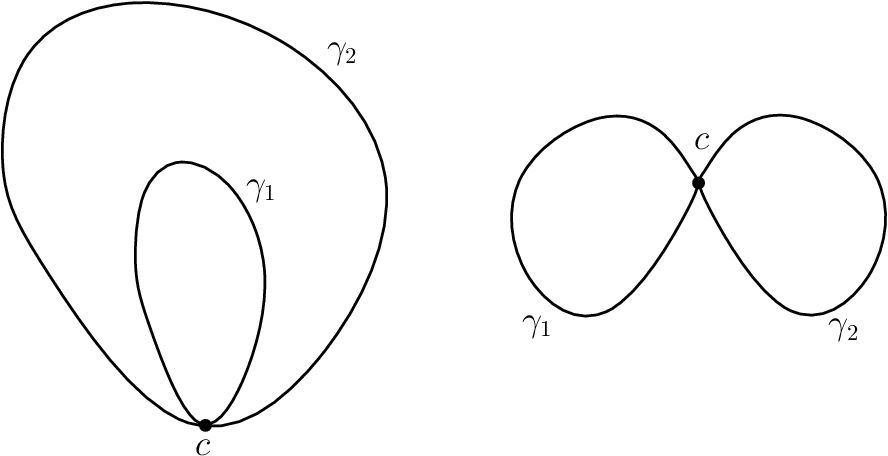}
	\caption{The two a priori arrangements for the preimages $\gamma_1$ and $\gamma_2$ of $\partial U_d$, which both intersect at the critical point $c$. The proof of Lemma \ref{lemma:lobes} demonstrates that the nested arrangement (left) leads to a contradiction, forcing the preimages to form a figure-eight curve (right).}
	\label{Figure:ocho}
\end{center}
\end{figure}

 Assume that the configuration of $\gamma_1$ and $\gamma_2$ resembles the left side in Figure \ref{Figure:ocho}. Let $B_1$ be the bounded component of $\gamma_1$, and denote by $E_1$ and $E_2$ the unbounded components of $\gamma_1$ and $\gamma_2$, respectively. Then, since $\gamma_1$ and $\gamma_2$ are mapped one-to-one to $\partial U_d$, two possibilities emerge:

\begin{enumerate}
    \item Assume $B$ bijectively maps both the bounded region $B_1$ and the unbounded region $E_2$ onto $U_d$. Since $E_2$ is unbounded, it must contain $\infty$. Because $B(\infty)=\infty$ and $B(E_2) = U_d$, this would imply $\infty \in U_d$. However, $U_d$ is a bounded Jordan domain. This is a contradiction. Thus, this scenario is impossible.

    \item Assume $B$ maps the region $E_1\setminus \overline{B}_1$ 2-to-1 to $U_d$, which implies that $B$ maps the bounded region $B_1$ to $\widehat{\mathbb{C}}\setminus \overline{U}_d$. The image $B(B_1)$ therefore contains $\infty$.
    
     The preimages of $\infty$ under $B$ are $\infty$ itself and the pole $1/\bar{a}$. Since $B_1$ is a bounded set, it cannot contain $\infty$, so it must contain the pole: $1/\bar{a} \in B_1$.
    
     However, the boundary curve $\gamma_1 = \partial B_1$ lies entirely in $\mathbb{C}\setminus\overline{\mathbb{D}}$. Since $B_1$ is the bounded region enclosed by $\gamma_1$, $B_1$ itself must also be contained entirely in $\mathbb{C}\setminus\overline{\mathbb{D}}$. This follows because if $B_1$ intersected the closed unit disk $\overline{\mathbb{D}}$, its boundary $\gamma_1$ would necessarily intersect $\overline{\mathbb{D}}$ as well, contradicting our premise that $\gamma_1$ lies entirely outside the closed disk.
    
     We have now reached a contradiction: the pole $1/\bar{a}\in\mathbb{D}$ must belong to $B_1$, while $B_1 \subset \mathbb{C}\setminus\overline{\mathbb{D}}$. Therefore, this scenario is also impossible.
\end{enumerate}

Consequently, both bounded components $B_1$ and $E_2$ are mapped onto $U_d$, confirming that the arrangement of $\gamma_1$ and $\gamma_2$ must be as shown to the right in Figure \ref{Figure:ocho}. In other words, there exists a connected component of $B^{-1}(\partial U_d)$ that forms a figure-eight curve.

Since the critical point c has multiplicity 2, then each sector depicted in purple around the critical point $c$ in Figure \ref{Figure:sectors}, is mapped one-to-one under $B$ onto the purple sector around $z_d$. Therefore, each of the lobes of the figure-eight curve must lie on opposite sides of the straight line formed by $0$, $c$, and $a$.
\end{proof}

\begin{figure}
\begin{center}
	\includegraphics[scale=0.9]{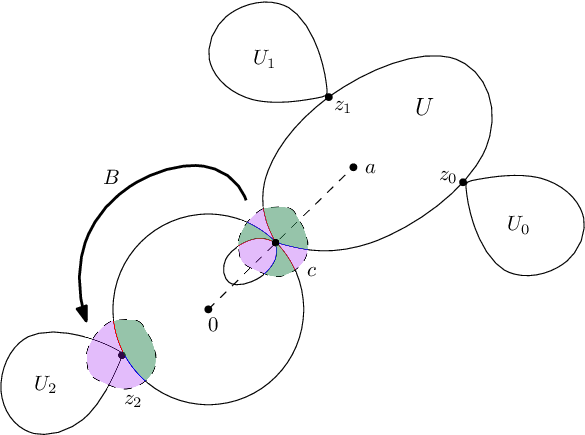}
	\caption{$B$ maps the purple sectors around $c$ onto a single purple sector around the critical value $z_2$. This mapping forces the two lobes of the figure-eight curve $\gamma = B^{-1}(\partial U_d)$ intersecting at $c$ to lie on opposite sides of the broken line connecting the origin, $c$, and $a$.}
	\label{Figure:sectors}
\end{center}
\end{figure}

For $i=\underline{0}, \underline{1},\dots  \underline{d}$, let $U_{id}$ denote the connected components of $B^{-1}(U_d)$. By Lemma \ref{lemma:lobes}, one of the connected components of $B^{-1}\left(\partial U_d\right)$ is a figure-eight curve, with one lobe on each side of the straight line joining $0$ and $a$. More precisely, $U_{\underline{0}d}$ is the bounded component of the lobe of $B^{-1}\left(\partial U_d\right)$ lying on the right-hand side of the segment from $0$ to $a$, and $U_{\underline{d}d}$ is the bounded component of the lobe on the left-hand side. The remaining preimages $U_{id}$ for $i=\underline{1},\dots  \underline{d-1}$ are those connected components of $B^{-1}(U_d)$ for which $\overline{U}_{id}\cap \left(\overline{U}\setminus \{c\}\right)\neq \emptyset$, as shown in Figure \ref{fig:Step02}. The indices then increase along $\partial U$ following the positive orientation, starting from its root $c$.

\begin{figure}
\begin{center}
	\includegraphics[scale=0.8]{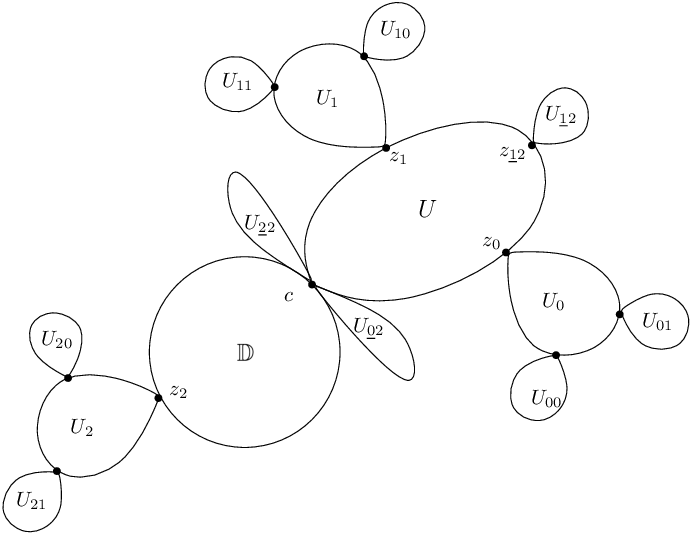}
	\caption{Sketch of the drops and roots of level 1 for the $d=2$ case. The preimages of the level 0 drop $U$ are shown: $U_0$ and $U_1$ are attached to $\partial U$ at their roots $z_0$ and $z_1$, while $U_2$ is attached to $\mathbb{S}^1$ at its root $z_2$ (which is also the critical value $v$).}
	\label{fig:Step02}
\end{center}
\end{figure}

\vspace{0.5cm}
\noindent\textbf{Drops of level $k\geq 3$}\label{sec:homeo_super_labeling_k3}

Now, we extend the labeling to drops $U_{i_1\dots i_k}$, with $k\geq 3$ and $i_1\dots i_k\in\Xi_{\mathcal{A}_d}$.
For $i_1\dots i_k\in \Xi_{\mathcal{A}_d} $, we recursively define $U_{i_1i_2\dots i_k}$ as follows, assuming that we have already defined all drops up to level $k-1$:
\begin{enumerate}
	\item If $i_{k-1}\in \left\{0,1,\dots , d\right\}$, then
	$U_{i_1i_2\dots i_k}$ is the connected component of $B^{-1}\left(U_{i_2\dots i_{k}}\right)$ such that $\overline{U}_{i_1\dots i_k}\cap\overline{U}_{i_1\dots i_{k-1}}\neq\emptyset $. This labeling follows the positive orientation of $\partial U_{i_1\dots i_{k-1}}$, starting at its root. 
	
	\item If $i_{k-1}\in \left\{\underline{0}, \underline{1}, \dots \underline{d}\right\}$, then $U_{i_1\dots i_k}$ denotes the connected component of $B^{-1}\left(U_{i_2\dots i_k}\right)$ such that $\overline{U}_{i_1\dots i_k}\cap\overline{U}_{i_1\dots i_{k-2}}\neq\emptyset $. This itinerary follows the positive orientation of $\partial U_{i_1\dots i_{k-2}}$, starting at its root. Specifically, $U_{i_1\dots i_k}$ is the $(k-2)$-th preimage of $U_{i_{k-1}}$ with its root on $\partial U_{i_1\dots i_{k-2}}$.
\end{enumerate}

In general, we denote the root of $U_{i_1\dots i_k}$ as $z_{i_1\dots i_k}$. The labels defined by the admissibility rules establish a bijection between the set of admissible words of length $k$ and the set of drops of level $k$.

\begin{lemma}
Let $U_{i_1\dots i_k}$ be a drop of level $k\geq 0$, and let $z$ be its root. Then, either $z\in\mathbb{S}^1$, or $z$ lies on the boundary of a drop of level $0\leq l< k$.
\end{lemma}
\begin{proof}
Let $s$ be the smallest non-negative integer such that $B^s(z)\in\mathbb{S}^1$. If $z$ is not on $\mathbb{S}^1$, then $s>0$, and $B^{s-1}(z)\in\partial U\setminus \{c\}$. Let $l=s-1<k$. Therefore, $z$ belongs to the boundary of a drop of level $l$.
\end{proof}

The action of $B$ on a drop in level $k$ translates to a right-hand shift in the labels,
\[
	B\left(U_{i_1i_2\dots i_k}\right)=\left\{\begin{array}{lr}
	U_{i_2\dots i_k}, & \mbox{ if } k\geq 2,\\
	U, & \mbox { if } k=1,
	\end{array}\right.
\]
for any word $i_1\dots i_k \in \Xi_{\mathcal{A}_d}$.

\begin{remark}
\label{remark:super-attracting-cycle-roots}
We are considering the super-attracting scenario where the critical point $c \in \mathbb{S}^1$ belongs to the 2-cycle $\{c, z_d\}\subset \mathbb{S}^1$. Because $B|_{\mathbb{S}^1}$ is a homeomorphism, the only preimages of $c$ that lie on $\mathbb{S}^1$ must belong to this cycle. Specifically, for any $k \ge 0$, the set $B^{-k}(c) \cap \mathbb{S}^1$ contains only $c$ if $i_k$ is even and only $z_d$ if $k$ is odd.

The set 
\begin{equation} \label{eq:S_Ad_definition}
	\mathcal{S}_{\mathcal{A}_d} = \left\{ i_1\dots i_k\in \Xi_{\mathcal{A}_d}\colon i_j \in \left\{d, \underline{0}, \underline{d} \right\}, \ j\in \{1,\dots k\} \right\}, 
\end{equation}
identifies the admissible words whose corresponding root $z_{i_1\dots i_k}$ is one of the two points in the cycle $\{c, z_d\}$. That is, a root $z_{i_1\dots i_k}$ lies on $\mathbb{S}^1$ if and only if the word $i_1\dots i_k$ belongs to $\mathcal{S}_{\mathcal{A}_d}$, and this root must be either $c$ or $z_d$.
\end{remark}

\begin{figure}
\begin{center}
	\includegraphics[scale=1.2]{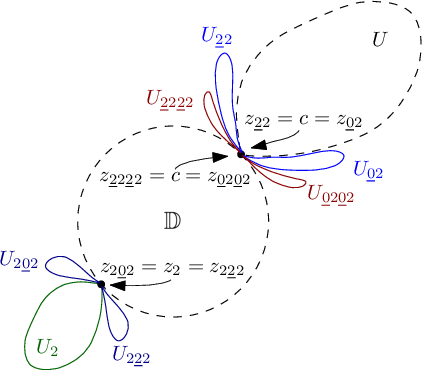}
	\caption{Sketch of drops whose roots alternate between the points $c$ and $z_2$ of the 2-cycle for $d=2$. Level 1 drop ($U_2$, in green) and level 3 drops ($U_{2\underline{0}2}$ and $U_{2\underline{2}2}$, in dark blue) are rooted at $z_2$. Level 2 drops ($U_{\underline{0}2}$ and $U_{\underline{2}2}$, in blue) and level 4 drops ($U_{\underline{0}2\underline{0}2}$ and $U_{\underline{2}2\underline{2}2}$, in red) are rooted at $c$.}
	\label{fig:multipleroots}
\end{center}
\end{figure}

\begin{remark}
\label{remark:biholomorphism}
Let $k\geq 2$. Then, the drops of level $k$ with labels in $\Xi_{\mathcal{A}_d}\setminus \mathcal{S}_{\mathcal{A}_d}$ do not intersect the forward orbit of $c$. Therefore,
\[
	B|_{\overline{U}_{i_1i_2\dots i_k}}\colon \overline{U}_{i_1\dots i_k}\to \overline{U}_{i_2\dots i_k},
\]
is a biholomorphism for $i_1\dots i_k \in \Xi_{\mathcal{A}_d}\setminus \mathcal{S}_{\mathcal{A}_d}$.
\end{remark}

\noindent\textbf{Limbs. }
For each $j\in\left\{0, d-1\right\}$, let $j^k$ represent the word of length $k$ with $j$ in each entry, and let $\overline{j}$ denote the infinite word with $j$ in each entry. Now, consider the following sets that we will refer to as \emph{limbs}:
\begin{eqnarray*}
	L_{\overline{j}} &=& \overline{U\cup \bigcup_{k\geq 1} U_{j^k}}= \bar{U}\cup\bar{U}_{j}\cup\bar{U}_{jj}\cup\cdots .
\end{eqnarray*}

We will say that the limb $L_{\overline{j}}$ \emph{starts at $U$}, in the sense that it is the drop in $L_{\bar{j}}$ of the smallest level. For $i= 0,1,\dots d$, we denote by $L_{i \cdot \overline{j}}$ each of the connected components of $B^{-1}\left(L_{\overline{j}} \right)$ located in $\mathbb{C}\setminus\mathbb{D}$, which start at $U_{i}$.

In general, for each $k\in\mathbb{N}$ and $i_1\dots i_k\in\Xi_{\mathcal{A}_d}$, we recursively define $L_{i_1\cdots i_k \cdot \overline{j}}$ as the connected component of $B^{-1}\left(L_{i_2\cdots i_{k}\cdot \overline{j}}\right)$ that starts at $U_{i_1\dots i_k}$.
\begin{figure}
\begin{center}
	\includegraphics[scale=0.7]{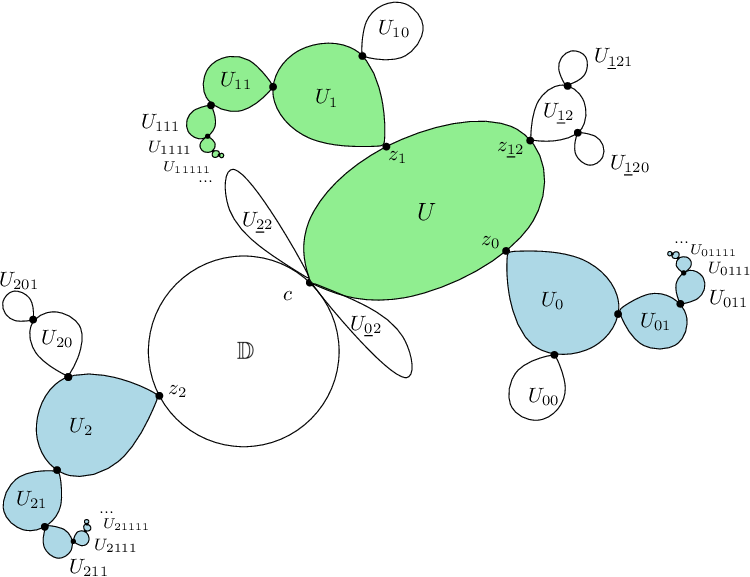}
	\caption{Sketch of the structure of the limb $L_{\overline{1}}$ for $d=2$. The limb $L_{\overline{1}}$ (green) is depicted. Also shown (blue) are two of its preimages under $B$: $L_{0\cdot \overline{1}}$, originating from drop $U_0$, and $L_{2\cdot \overline{1}}$, originating from drop $U_2$. The third preimage, $L_{1\overline{1}}$ (originating from $U_1$), is contained within $L_{\overline{1}}$ itself. Note that $B$ maps both blue limbs onto $L_{\overline{1}}$. Furthermore, $L_{2\cdot \overline{1}}$ is disjoint from both $L_{\overline{1}}$ and $L_{0\cdot \overline{1}}$, whereas $L_{0\cdot \overline{1}}$ intersects $L_{\overline{1}}$ only at the root $z_0$, the point where drop $U_0$ attaches to drop $U$.}
\end{center}
\end{figure}

\begin{lemma}[Diameter of $\overline{U}_{j^k}$]
\label{lemma:diamUto0}
	Let $j\in\{0, d-1\}$, and consider the drop of level $k$ given by $U_{j^k}\subset L_{\overline{j}}$. Then
	\[
		\operatorname{diam}\left(\overline{U}_{j^k}\right) \to 0\ \mbox{ as }\ k\to \infty.
	\]
\end{lemma}
\begin{proof}
	Fix $j\in \{0, d-1\}$. The Remark \ref{remark:biholomorphism} justifies the existence of an analytic inverse of $B$, denoted by $\phi_j$, satisfying $B\circ \phi_j|_{\overline{U}_j} = \mbox{Id}|_{\overline{U}_j}$. Let $\overline{U}_{j^k}:= \phi^{k-1}_j\left(\overline{U}_j\right)$. Since the compact set $\overline{U}_j$ is disjoint from the forward orbit of the critical point , which is also a compact set, there exists some $\delta > 0$ such that the $\delta$-neighborhood of $\overline{U}_j$, denoted by $D$, is also disjoint from the forward orbit of the critical point. Moreover, $D$ is simply connected and intersects the Julia set since $B^2\left(\overline{U}_j\right)= \overline{\mathbb{D}}$ and $\mathcal{A}(0)\subset \overline{\mathbb{D}}$. Hence, by Lemma \ref{lemma:diam0} we can conclude that
	\[
		\operatorname{diam}\left(\overline{U}_{j^k}\right) \to 0\ \mbox{ as }\ k\to \infty.
	\]
\end{proof}

\begin{remark}
\label{remark:geometric_convergence}
The proof of Lemma \ref{lemma:diamUto0} establishes that the inverse branch $\phi_j$ is an analytic map on a simply connected domain $D$ containing $ \overline{U}_j$ that is disjoint from the postcritical set. The family of iterates $\{\phi_j^n\}$ is normal on $D$, and every limit function is constant.

By the Schwarz-Pick Lemma, any such map is a strict contraction with respect to the Poincaré metric on $D$. On the compact set $K = \overline{U}_j$, this implies uniform contraction. By the equivalence of the Poincaré and Euclidean metrics on a compact set, this means there exists a constant $\lambda \in (0, 1)$ and a constant $C$ such that:
\[
    \text{diam}(\overline{U}_{j^k}) = \text{diam}(\phi_j^{k-1}(\overline{U}_j)) \leq C \cdot \lambda^{k-1}
\]
This geometric convergence ensures that the series $\sum_{k=1}^{\infty} \text{diam}(\overline{U}_{j^k})$ is convergent.
\end{remark}

\begin{lemma}
\label{lemma:intersection_is_point}
Let $j\in\{0,d-1\}$. Then, the intersection
\begin{equation}
\label{eq:FixedPointInBoundary}
	\bigcap_{k=1}^\infty L_{j^k\cdot \overline{j}}
\end{equation}
is a single point, denoted as $\zeta_j $, which is fixed under $B$ and lies on $\mathbb{C}\setminus \overline{\mathbb{D}}$.
\end{lemma}

\begin{proof}
By definition, the limb segment $L_{j^k\cdot \overline{j}}$ is the closure of the union of nested drops starting from level $k$: $L_{j^k\cdot \overline{j}} = \overline{\bigcup_{l=k}^{\infty} U_{j^l}}$. Each $L_{j^k\cdot \overline{j}}$ is a closed and connected set. The construction ensures this is a nested sequence: $L_{j^{k+1}\cdot \overline{j}} \subset L_{j^k\cdot \overline{j}}$ for all $k \ge 1$.
The diameter of the limb segment is bounded above by the sum of the diameters of the drops it contains:
\[
    \text{diam}(L_{j^k\cdot \overline{j}}) \leq\sum_{l=k}^{\infty} \text{diam}(\overline{U}_{j^l}).
\]
As established in Remark \ref{remark:geometric_convergence}, the diameters of the drops $\text{diam}(\overline{U}_{j^l})$ converge to zero geometrically. This guarantees that the series $\sum_{l=1}^{\infty} \text{diam}(\overline{U}_{j^l})$ is convergent. A necessary property of any convergent series is that its tail must tend to zero. Therefore,
\[
    \lim_{k \to \infty} \sum_{l=k}^{\infty} \text{diam}(\overline{U}_{j^l}) = 0.
\]
By the Squeeze Theorem, $\lim_{k \to \infty} \text{diam}(L_{j^k\cdot \overline{j}}) = 0$. We now have a nested sequence of non-empty compact sets $\{L_{j^k\cdot \overline{j}}\}_{k=1}^{\infty}$ whose diameters tend to zero. By Cantor's Intersection Theorem, their intersection must be a single point, which we denote $\{\zeta_j\}$.

Moreover, $B$ maps drops to drops via $B(U_{j^{k+1}}) = U_{j^k}$, which implies it maps limb segments to limb segments: $B(L_{j^{k+1}\cdot \overline{j}}) = L_{j^k\cdot \overline{j}}$ for $k \ge 1$. Applying $B$ to the intersection:
\[
    B\left(\bigcap_{k=1}^{\infty}L_{j^k\cdot \overline{j}}\right) = B\left(\bigcap_{k=2}^{\infty}L_{j^k\cdot \overline{j}}\right) \subseteq \bigcap_{k=2}^{\infty} B(L_{j^k\cdot \overline{j}}) = \bigcap_{k=2}^{\infty}L_{j^{k-1}\cdot \overline{j}} = \bigcap_{l=1}^{\infty}L_{j^{l}\cdot \overline{j}}.
\]
This shows $B(\{\zeta_j\}) = \{\zeta_j\}$, so $B(\zeta_j) = \zeta_j$.
Finally, since each set $L_{j^k\cdot \overline{j}}$ is contained in $\mathbb{C} \setminus \overline{\mathbb{D}}$ (as they are unions of drops outside the closed unit disk), their intersection point $\zeta_j$ must also lie in $\mathbb{C} \setminus \overline{\mathbb{D}}$.
\end{proof}

\begin{lemma}
\label{lemma:FixedPointsOnInfinity}
	Let $j\in\{0,d-1\}$. Then, the fixed points $\zeta_0,\zeta_{d-1}$ given in \eqref{eq:FixedPointInBoundary} belong to $\partial\mathcal{A}^*(\infty )$. Moreover, each of these points corresponds to the landing point of a fixed B\"ottcher ray of $\mathcal{A}^*(\infty)$. 
\end{lemma}
\begin{proof}
     
    The fixed points $\zeta_0,\zeta_{d-1}$ given in \eqref{eq:FixedPointInBoundary} necessarily coincide with two of the $2d+2$ fixed points in Proposition \ref{prop:boundary_dynamics_and_fixed_points}. Given that $\zeta_j\in \mathbb{C}\setminus\overline{\mathbb{D}}$, it follows that $\zeta_j$ must belong to $\partial\mathcal{A}^*(\infty)$. In other words, 
    \[
        \left\{\zeta_0,\zeta_{d-1}\right\}=\left\{\phi^{-1}\left(0\right), \phi^{-1}\left((d-1)/d\right)\right\}.
    \]
    
    Moreover, these points correspond respectively to the landing points of the fixed B\"ottcher rays $R^\infty _{\theta_j}(t)$ on $\mathcal{A}^*\left(\infty\right)$, with angles $\theta _j= j/d$, where $j=0,d-1$.
\end{proof}

\noindent\textbf{Puzzle Pieces Construction. }
We can construct a curve $\Gamma$ that connects the fixed points $\zeta_0$ and $\zeta_{d-1}$ through $L_{\overline{d-1}}\cup L_{\overline{0}}$. Denote as $a_{0^k}$ the $k$-th preimage of $z=a$ that lies on $U_{0^k}$, and similarly, denote as $a_{(d-1)^k}$ the $k$-th preimage of $z=a$ lying on $U_{(d-1)^k}$. Then, the construction of $\Gamma $ is as follows:
\begin{enumerate}
    \item \label{item:gamma_step1} Take the preimage of the segment connecting $0$ to $a$ that passes through $c$ and lies on $L_{\overline{d-1}}\cup L_{\overline{0}}$. This is a segment connecting $a_0$ and $a_1$ through $z_0$, $a$ and $z_{d-1}$.
    \item Next, take the preimage of the segment connecting $a$ to $a_{d-1}$ through $z_{d-1}$ on $L_{\overline{d-1}}\cup L_{\overline{0}}$. This is a segment connecting $a_{d-1}$ with $a_{(d-1)(d-1)}$ through $z_{(d-1)(d-1)}$. Similarly, the preimage of the segment connecting $a$ to $a_0$ through $z_0$ on $L_{\overline{d-1}}\cup L_{\overline{0}}$ is a segment that connects $a_0$ with $a_{00}$ through $z_{00}$.
    \item By joining the segments obtained in steps 1 and 2, we obtain a curve connecting $a_0$ to $a_{d-1}$ on $L_{\overline{d-1}}\cup L_{\overline{0}}$.
    \item For $k\in\mathbb{N}$,  take the preimage of the segment connecting $a_{(d-1)^k}$ to $a_{(d-1)^{k+1}}$ through $z_{(d-1)^{k+1}}$ on $L_{\overline{d-1}}\cup L_{\overline{0}}$. This is a segment connecting $a_{(d-1)^{k+1}}$ with $a_{(d-1)^{k+2}}$ through $z_{(d-1)^{k+2}}$, and due to Lemma \ref{lemma:diamUto0}, it
    converges to $\zeta_{d-1}$ as $k\to\infty $. Similarly, the preimage of the segment connecting $a_{0^k}$ to $a_{0^{k+1}}$ through $z_{0^{k+1}}$ on $L_{\overline{d-1}}\cup L_{\overline{0}}$ is a segment connecting $a_{0^{k+1}}$ with $a_{0^{k+2}}$ through $z_{0^{k+2}}$, and converges to $\zeta_0$ as $k\to\infty$.
    \item Finally, there is a curve, $\Gamma$, connecting $\zeta_{d-1}$ and $\zeta_0$ on $L_{\overline{d-1}}\cup L_{\overline{0}}$ formed by the union of all segments obtained in step 4.
\end{enumerate}

Furthermore, we define the extended curve $\tilde{\Gamma} = \Gamma \cup R^\infty_0 \cup R^\infty_{(d-1)/d}$. This $\tilde{\Gamma}$ is a simple closed curve on the Riemann sphere $\widehat{\mathbb{C}}$ (passing through $\infty$). By the Jordan Curve Theorem, $\tilde{\Gamma}$ partitions the sphere into two disjoint open half-planes. We denote the half-plane that contains the unit circle $\mathbb{S}^1$ as $H_{\mathbb{S}^1}$.

From now on, our attention will be directed towards $H_{\mathbb{S}^1}$, because any bi-accessibility must occur on the unit circle due to the symmetry of $B$ under $\mathcal{I}$.

We fix the normalization of the Böttcher coordinates from $\mathcal{A}^*(\infty)$ such that $R^\infty_0$ lands at $\zeta_0$ and $R^\infty_{(d-1)/d}$ lands at $\zeta_{d-1}$. We define $H_{\mathbb{S}^1}$ as the half-plane intersected by all Böttcher rays $R^\infty_\theta$ with angles $\theta \in ((d-1)/d, 1)$. See Figure \ref{fig:W}. 

 By Theorem \ref{thm:Goldberg}, the map $m_{d+1}$ admits a 2-cycle $\{t_1,t_2\}$, where
 \[
    t_1=\frac{d^2+d-1}{d(d+1)}\qquad t_2=\frac{d^2+2d-1}{d(d+1)}, 
 \] 
 contained in the interval of fixed points $I=\left[\frac{d-1}{d},1\right]$. Consequently, the map $B$ has a corresponding 2-cycle $\{\zeta_{t_1},\zeta_{t_2}\}$, which lies on the arc of $\partial\mathcal{A}^*(\infty)$ between the fixed points $\zeta_{0}$ and $\zeta_{d-1}$.

Under the action of $\mathcal{I}(z)$, we have a symmetric ray $R^0_{s_1}$ landing on $\partial\mathcal{A}^*(0)$ at the point $\zeta_{s_1} = \mathcal{I}(\zeta_{t_1})$, where
\[
    s_1=\frac{d+1}{d(d+2)},
\]
and $R^0_{s_2}$ landing on $\mathcal{I}(\zeta_{t_2})$, where
\[
    s_2=\frac{1}{d(d+2)}. 
\]

\begin{lemma}
\label{lemma:final-puzzlepiecesparticular}
    Fix $d\geq 1$. Then, the period $2$ point $\zeta_{t_1}\in\partial \mathcal{A}^*(\infty)$ with
    \[
        t_1=\frac{d^2+d-1}{d(d+1)}
    \]
    that is the landing point of the B\"{o}ttcher ray $\widehat{R}^\infty_{t_1}$, is bi-accessible.
\end{lemma}

We will now construct puzzle pieces to show that $\zeta_{t_1} = \mathcal{I}(\zeta_{s_1})$, which proves bi-accessibility.

    We will refer to the sets $\widehat{R}^\infty _{0} = R^\infty_0 \cup L_{\overline{0}}$ and $\widehat{R}^\infty _{(d-1)/d} = R^\infty_{(d-1)/d} \cup L_{\overline{d-1}}$ as \emph{extended B\"ottcher rays}, or simply as \emph{extended rays}. Due to Lemma \ref{lemma:FixedPointsOnInfinity}, we know that these are connected compact sets. 
    
    Let $E$ denote the equipotential in $\mathcal{A}^*(\infty)$ given by $\phi^{-1}\left(\left\{z\colon |z|=r\right\}\right)$, where $r<1$ is fixed. Consider the set 
    \[
        \mathbb{C}\setminus \left(\widehat{R}^\infty_{0}\cup \widehat{R}^\infty_{(d-1)/d}\cup E\cup \mathbb{D}\right),
    \]
which consists of two bounded connected components. Let $W^\infty$ denote the bounded connected component for which the intersection with the set of rays 
\[
\left\{R^\infty_\theta \colon \theta \in ((d-1)/d,1)\right\}
\] 
is nonempty. Let $W^0$ be the symmetric image of $W^\infty$ under $\mathcal{I}$, and denote as $W$ the union $W^\infty\cup W^0$, as illustrated in Figure \ref{fig:W}.

\begin{figure}
\begin{center}
	\includegraphics[width=1\linewidth]{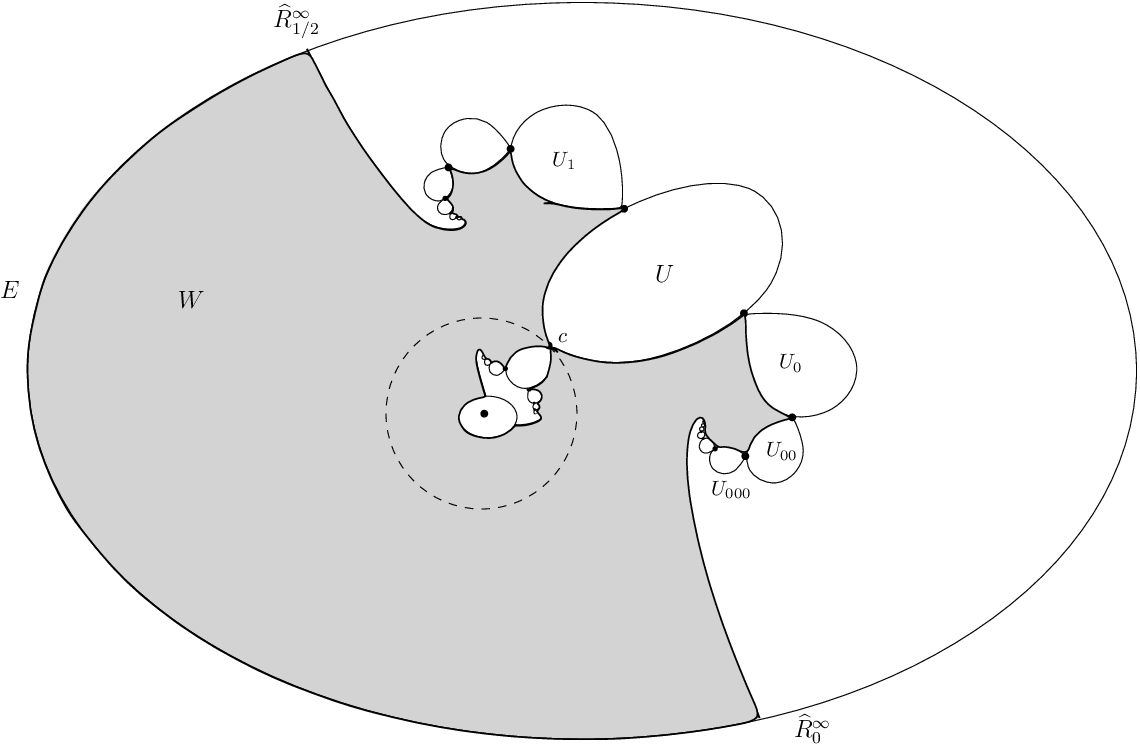}
	\caption{A sketch of the region $W$ (in gray) for the case $d=2$. $W$ is the union of $W^\infty$ and its symmetric image $W^0$. The region $W^\infty$ is bounded by the extended B\"ottcher rays $\widehat{R}^\infty_{0}$ and $\widehat{R}^\infty_{1/2}$, and the equipotential $E$. The unit circle is shown as a dashed line.}
	\label{fig:W}
\end{center}
\end{figure}

Since $\widehat{R}^\infty_{0}$ is connected, its preimages under $B$ in $\mathbb{C}\setminus\mathbb{D}$ consist of $d+1$ disjoint connected components formed by the preimages of $L_{\overline{0}}$ and $R^\infty_0$ in $\mathbb{C}\setminus\mathbb{D}$: 
\[
    \widehat{R}^\infty_{j/(d+1)}\colon= R^\infty_{j/(d+1)}\cup L_{j\cdot\overline{0}}, \qquad j\in\{0,\dots, d\}.
\]
Similarly, the preimages of $\widehat{R}^\infty_{(d-1)/d}$ are given by:
\[
    \widehat{R}^\infty_{\frac{d(j+1)-1}{d(d+1)}}\colon= R^\infty_{\frac{d(j+1)-1}{d(d+1)}}\cup L_{j\cdot \overline{d-1}},\qquad j\in\{0\dots, d\}.
\]
In general, we will denote the preimages of $\widehat{R}^\infty_\theta$ as $\widehat{R}^\infty_\varphi$, where $m_2(\varphi)=\theta$.

We will focus on the preimages contained in the interval $I$, which are 
\[
\widehat{R}^\infty_{0},\quad  \widehat{R}^\infty_{\frac{d-1}{d}},\quad \widehat{R}^\infty_{\frac{d}{d+1}},\mbox{ and } \widehat{R}^\infty_{\frac{d^2+d-1}{d(d+1)}}.
\]

Denote by $P$ the connected component of \label{def:base_puzzle_P}
\[
	W\setminus \left(\widehat{R}^\infty_{d/(d+1)}\cup \widehat{R}^0_{1/(d+1)}\right)
\]
that intersects the B\"ottcher rays of $\mathcal{A}^*(\infty)$ with angles in $\left(\frac{d-1}{d},\frac{d}{d+1}\right)$. 

Since $\{c,z_2\}$ is a super-attracting cycle, for a sufficiently small $\epsilon >0$, there is a disk $D_c=D(c,\epsilon)$ that is contained in $\mathcal{A}^*(c)$, and $\overline{D}_c\subset B^{-2}\left(D_c\right)\cap\mathcal{A}^*(c)$. Denote as $D_{z_2}$ the preimage of $D_c$ under $B$ that is contained in $\mathcal{A}^*(z_2)$. Then, it is also true that $D_{z_2}\subset B^{-2}\left(D_{z_2}\right)\cap\mathcal{A}^*(z_2)$. We can now define the \emph{puzzle piece} $P_0$ as the open set
\[
	P_0:=P\setminus \left(\overline{D}_c\cup \overline{D}_{z_2}\right).
\]

Let $P_1$ be the connected component of the preimage of $P_0$ under $B$ that is contained in $W$. It is important to note that, by construction, $P_1$ intersects the B\"ottcher rays of $\mathcal{A}^*(\infty)$ with angles in $\left(\frac{d^2+d-1}{d(d+1)},\frac{d^2+2d}{(d+1)^2}\right)$. Additionally, by removing the disks $D_c$ and $D_{z_2}$, we ensure that $\overline{P}_0\cap \overline{P}_1=\emptyset$.

We define $P_1 = B^{-1}(P_0) \cap W$ and $P_2 = B^{-1}(P_1) \cap W$.
Due to the expansion of $B^{-1}$ near the cycle, we have the strict inclusion $\overline{P}_2 \subset P_0$ (Figure \ref{fig:3Puzzles}).

\begin{figure}
\begin{center}
	\includegraphics[width=1\linewidth]{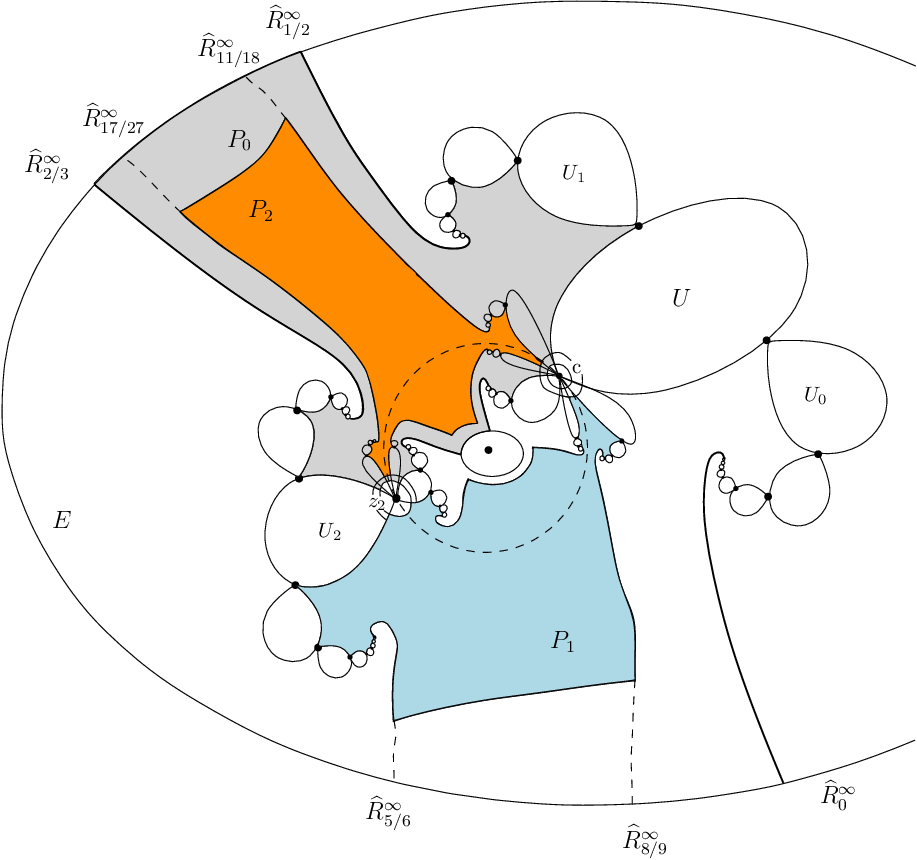}
	\caption{Super-attracting homeomorphism scenario: A sketch of the regions $P_0$ (in gray), $P_1$ (in blue), and $P_2$ (in orange). The construction, which removes disks around the 2-cycle, ensures that $\overline{P_0} \cap \overline{P_1} = \emptyset$. The preimage $P_2 = B^{-1}(P_1)$ is strictly contained within $P_0$ ($\overline{P_2} \subset P_0$), guaranteeing a contraction}
	\label{fig:3Puzzles}
\end{center}
\end{figure}
Take $P_2$ as the connected component of $B^{-1}(P_1)$ that is contained in $W$. \label{def:P2_first_construction}. Given that $P_2$ intersects the B\"ottcher rays of $\mathcal{A}^*(\infty)$ with angles in $\left(\frac{d^3+d^2-1}{d(d+1)^2},\frac{d^3+2d^2+d-1}{(d+1)^3}\right)$, and considering that $B^{-1}$ is expanding near the super-attracting cycle $\left\{c,z_2\right\}$, it follows that $\overline{P}_2\subset P_0$, as illustrated in Figure \ref{fig:3Puzzles}.

\paragraph{Attracting case}
We now examine the specific parameters $r=2d+1$, and $\alpha $ such that the critical point $c$ belongs to the basin of attraction of a $2$-cycle on $\mathbb{S}^1$ with rotation number $1/2$. We denote this cycle by $\{\xi_0,\xi_1\}$, with $\xi_0$ being the marked point.

Much of what we have constructed in the super-attracting case is applicable to this scenario as well. Consequently, we will only highlight the differences.

\vspace{0.5cm}
\noindent\textbf{Labeling drops and roots.} 
In this scenario, $z_2$ is not a critical value, requiring an adjustment in the labeling of the drops of level 2. We analyze the components of $B^{-1}(U_d)$, denoted by $U_{id}$, with $i\in \{0,\dots ,d\}$.

Since $B|_{\mathbb{S}^1}$ is injective and orientation-preserving, there is exactly one preimage of $U_d$ attached to the unit circle $\mathbb{S}^1$, while the remaining $d$ preimages are attached to the boundary of the drop $\partial U$. The labeling depends on the argument of the critical point $c$ relative to the reference angle $\xi_0$:

\begin{enumerate}
    \item \textbf{Case $\arg(c) < \arg(\xi_0)$:}
    The preimage attached to $\mathbb{S}^1$ lies to the right of the line segment from $0$ to $a$ and is labeled $U_{\underline{0}d}$. The remaining components attached to $\partial U$ are labeled $U_{id}$ for $i=\underline{1}, \dots, \underline{d}$ in increasing order relative to the orientation of $\partial U$.

    \item \textbf{Case $\arg(c) > \arg(\xi_0)$:}
    The preimage attached to $\mathbb{S}^1$ is labeled $U_{\underline{d}d}$. The remaining components attached to $\partial U$ are labeled $U_{id}$ for $i=\underline{0}, \dots, \underline{d-1}$ in increasing order.
\end{enumerate}

\begin{figure}
\centering
	\includegraphics[width=1\linewidth]{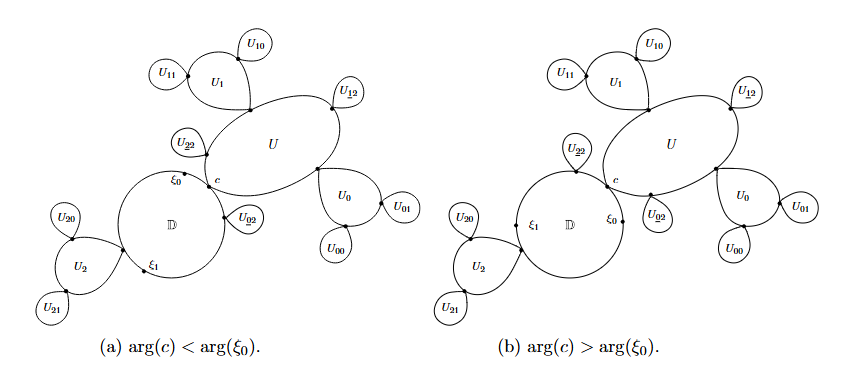}
\caption{Attracting homeomorphism scenario: The two possible configurations for the preimages of $U_2$, depending on the argument of the critical point $c$ relative to the attracting fixed point $\xi_0$. (a) If $\arg(c)<\arg(\xi_0)$, the preimage attached to $\mathbb{S}^1$ is $U_{\underline{0}2}$. (b) If $\arg(c)>\arg(\xi_0)$, the preimage attached to $\mathbb{S}^1$ is $U_{\underline{2}2}$.}
 \label{fig:Homeo-At-Step02}
\end{figure}

\vspace{0.5cm}
\noindent\textbf{Drops of level $k\geq 3$.}
Let $\mathcal{S}^0_{\mathcal{A}_d}$ represent the subset of $\mathcal{S}_{\mathcal{A}_d}$ (see \eqref{eq:S_Ad_definition}) defined as
\begin{equation}
\label{eq:setS0}
    \mathcal{S}^0_{\mathcal{A}_d}=\left\{i_1\dots i_k\in \Xi_\mathcal{A}\colon i_j\in \{\underline{0}, d\},\ j=1,\dots ,k\right\}
\end{equation}

Let $k\geq 3$. Given the definition of drops up to level $k-1$, the following rules apply:
\begin{enumerate}
    \item If $i_1\dots i_k\in\mathcal{S}^0_{\mathcal{A}_d}$, then $U_{i_1\dots i_k}$ is the preimage of $U_{i_2\dots i_k}$ attached to $\mathbb{S}^1$.
    \item If $i_1\dots i_k\not\in \mathcal{S}^0_{\mathcal{A}_d}$ and
    \begin{enumerate}
        \item $i_2\dots i_{k}\in\mathcal{S}^0_{\mathcal{A}_d}$, then $U_{i_1\dots i_k}$ is the preimage of $U_{i_2\dots i_k}$ attached to $U_{i_1\dots i_{k-3}}$ in the arc $(z_{i_1\dots i_{k-3}j},z_{i_1\dots i_{k-3}i_1})$, where $j=i_1-1\pmod{d}$. 
        \item $i_2\dots i_{k}\not\in\mathcal{S}^0_{\mathcal{A}_d}$, then $U_{i_1\dots i_k}$ is the preimage of $U_{i_2\dots i_k}$ attached to $U_{i_1\dots i_{k-1}}$.
    \end{enumerate}
   \end{enumerate}
As in the previous case, the action of $B$ on a drop of level $k$ results in a right-hand shift in the labels.

\vspace{0.5cm}
\noindent\textbf{Puzzle Pieces Construction. } 
We maintain the notation of the B\"ottcher rays and the region $W$ as previously established.

Since $\{\xi_0,\xi_1\}$ is an attracting cycle, there is $\epsilon>0$ such that $c\in \partial D_{0} =\partial \phi^{-1}\left(D(0,\epsilon )\right)$, where $\phi$ are the K\oe nigs coordinates, with $\overline{D}_{0}\subset \mathcal{A}^*(\xi_0)$. Note that $D_{0}$ does not intersect $\overline{U}$. Let $D_{2}$ be the second preimage of $D_{0}$ under $B$ such that $\overline{D}_{0}\subset \overline{D}_2\subset \mathcal{A}^*(\xi_0)$. It is worth observing that $z_{\underline{0}d}$ and $z_{\underline{d}d}$ belong to $\partial D_{2}$, but $D_2$ does not intersect $\overline{U}_{\underline{0}d}$ nor $\overline{U}_{\underline{d}d}$. 

We define the puzzle piece $P_0$ as the open set
\[
    P_0 := P\setminus \overline{D}_2,
\]
where $P$ is the base puzzle piece defined in Section \ref{def:base_puzzle_P} on page \pageref{def:base_puzzle_P}.
We define $P_1$ as the connected component of the preimage of $P_0$ under $B$ contained in $W$, as illustrated in Figure \ref{fig:3PuzzlesHomeo}. By removing the topological disk $D_2$, we ensure that $\overline{P}_0\cap \overline{P}_1=\emptyset$. $P_2$ is then taken as before.

\begin{figure}
\begin{center}
	\includegraphics[width=1\linewidth]{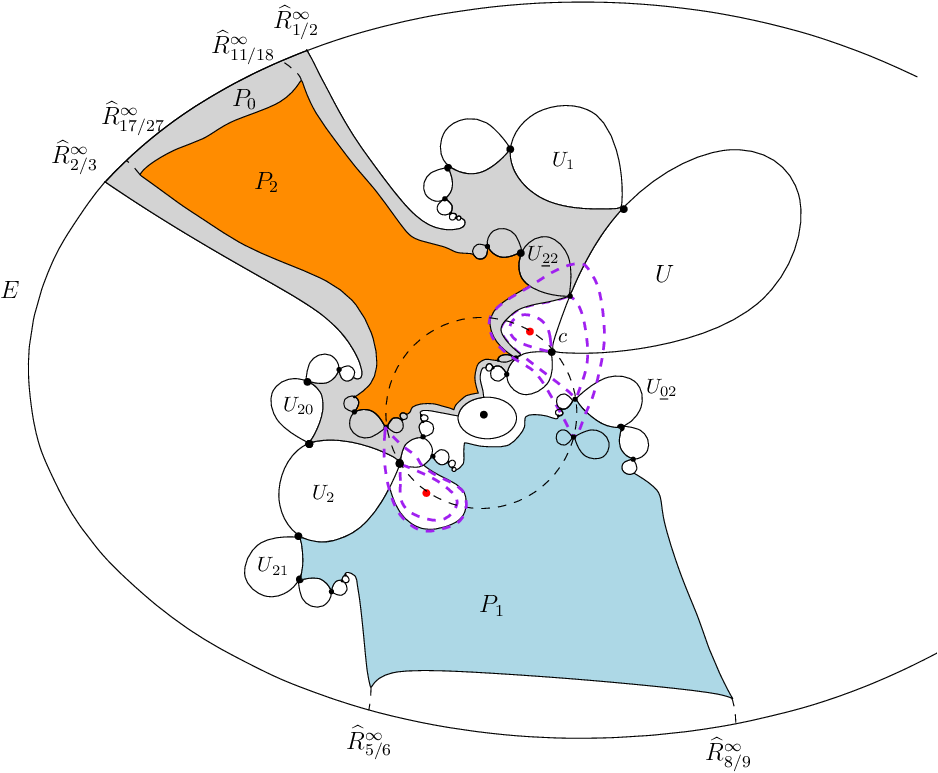}
	\caption{Attracting homeomorphism scenario: A sketch of the regions $P_0$ (in gray), $P_1$ (in blue), and $P_2$ (in orange). The figure also depicts the attracting cycle $\{\xi_0,\xi_1\}$ (in red) and the boundaries of the removed disks (in purple).}
	\label{fig:3PuzzlesHomeo}
\end{center}
\end{figure}

\subsection{The Endomorphism Case}
\label{subsec:EndoCase}

In this region, $B|_{\mathbb{S}^1}$ acts as a circle endomorphism with two simple critical points, denoted $c_+$ and $c_-$. We define their corresponding \emph{co-critical points} $c_+', c_-' \in \mathbb{S}^1$ by the condition $B(c_\pm') = B(c_\pm)$.

These points partition the unit circle into two intervals with distinct covering properties:
\begin{itemize}
    \item The interval $[c_+', c_-']$, which contains the critical points, maps 3-to-1 onto the critical value interval $[v_+, v_-]$.
    \item The complementary interval $[c_-', c_+']$, which is disjoint from the critical set, maps 1-to-1 onto $[v_-, v_+]$.
\end{itemize}
Additionally, the map $B|_{\mathbb{D}}$ reverses orientation on the arc $(c_-, c_+)$.

This mapping structure determines the location of the preimages of the critical points. Since the interval $[c_-', c_+']$ maps injectively to the critical values, it contains exactly one preimage for each critical point. The remaining preimages must therefore lie on the boundary of the drop $U$. We label these sets of preimages as follows:

\begin{itemize}
    \item \textit{On the Circle}: Let $z_d^\pm \in \mathbb{S}^1$ denote the unique preimages of $c_\pm$ lying on the unit circle.
    \item \textit{On the Drop Boundary}: Let $\{z_0^\pm, \dots, z_{d-1}^\pm\} \subset \partial U$ denote the remaining $d$ preimages of $c_\pm$.
\end{itemize}

\begin{figure}
\begin{center}
	\includegraphics[scale=1]{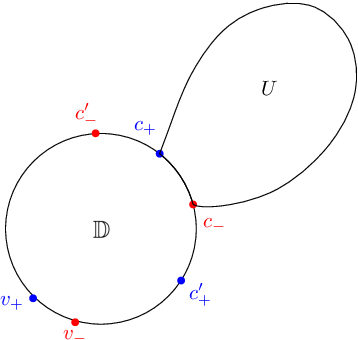}
	\caption{The dynamical configuration for the endomorphism case ($1<r<2d+1$). The map has two critical points ($c_+, c_-$) and two co-critical points ($c_+', c_-'$) in $\mathbb{S}^1$. The drop $U$ (preimage of $\mathbb{D}$) attaches to $\mathbb{S}^1$ on the arc $(c_-, c_+)$.}
	\label{fig:EndoGral}
\end{center}
\end{figure}

We focus on adjacent parameters where the Fatou domains contain both free critical points. We distinguish between two dynamical scenarios based on the type of cycle in $\mathbb{S}^1$: a super-attracting 2-cycle and an attracting 2-cycle.

We label the first preimages of the drop $U$. Specifically, we define $U_i$ as the unique connected component of $B^{-1}(U)$ such that its boundary contains the preimages $z_i^+$ and $z_i^-$.

Although the general labeling scheme remains consistent, the positioning of the preimages of the specific drop $U_d$ differs between the super-attracting and attracting cases. These differences need distinct constructions for the puzzle pieces in each scenario, which we detail below.

\paragraph{Super-attracting case}
Suppose $(r,\alpha)\in T_{1/2}$, with $1<r<2d+1$, and there exists a super-attracting $2$-cycle in $\mathbb{S}^1$. The critical points $c_+$ and $c_-$ are distinct. A cycle is super-attracting if and only if it contains a critical point. Therefore, the $2$-cycle must coincide with the forward orbit of exactly one of these critical points (e.g., $\{c_+, B(c_+)\}$). Consequently, only the critical value corresponding to that specific critical point lies on the cycle; the orbit of the other critical point remains independent.

\vspace{0.5cm}
\noindent\textbf{Labeling drops of level $k\geq 2$.}
Since the arc $[c_-', c_+']$ maps injectively to $[v_-, v_+]$, there exists a unique preimage $V$ of the drop $U_d$ attached to the unit circle $\mathbb{S}^1$. The labeling of $V$ depends on the specific super-attracting cycle configuration and its position relative to the line $L$ connecting $0$ and $a$ (see Figure \ref{fig:EndoSuper-At-Step02}).

\begin{enumerate}
    \item \textbf{Cycle $\{c_+, z_d^+\}$:} In this case, $z_d^+ = v_+$. The drop $V$ attaches to the arc $[v_-, c_+']$ (located to the right of $L$).
    \begin{itemize}
        \item We label $V = U_{\underline{0}d}$.
        \item The remaining preimages on $\partial U$ are labeled $U_{id}$ for $i = \underline{1}, \dots, \underline{d}$.
    \end{itemize}
    
    \item \textbf{Cycle $\{c_-, z_d^-\}$:} In this case, $z_d^- = v_-$. The drop $V$ attaches to the arc $[c_-', v_+]$ (located to the left of $L$).
    \begin{itemize}
        \item We label $V = U_{\underline{d}d}$.
        \item The remaining preimages on $\partial U$ are labeled $U_{id}$ for $i = \underline{0}, \dots, \underline{d-1}$.
    \end{itemize}
\end{enumerate}

\begin{figure}
\includegraphics[width=1\linewidth]{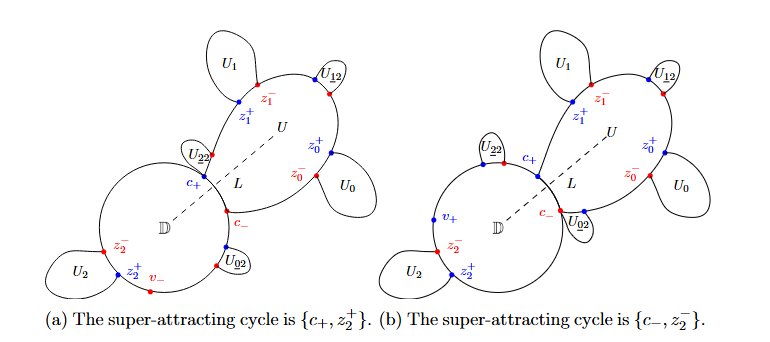}
\caption{Super-attracting endomorphism scenario for $d=2$: The two possible configurations for the preimages of $U_2$. The position of the unique preimage attached to $\mathbb{S}^1$ ($U_{\underline{0}2}$ or $U_{\underline{2}2}$) depends on which critical point is in the 2-cycle.}
 \label{fig:EndoSuper-At-Step02}
\end{figure}

Since the two cycle configurations are analogous, we assume without loss of generality that $\{c_+, z_2^+\}$ forms the super-attracting cycle (as illustrated in Figure \ref{fig:EndoSuper-At-Step02}).

For levels $k \ge 3$, the labeling of drops generally follows the scheme established for the attracting homeomorphism case. However, a modification is required for the second guideline to accommodate the specific geometry of the endomorphism. Assuming all drops up to level $k-1$ have been defined, we proceed with the following rules, recalling the definition of $\mathcal{S}^0_{\mathcal{A}_d}$ from \eqref{eq:setS0}.

\begin{enumerate}
    \item If $i_1\dots i_k\in\mathcal{S}^0_{\mathcal{A}_d}$, then $U_{i_1\dots i_k}$ is the preimage of $U_{i_2\dots i_k}$ attached to $\mathbb{S}^1$.
    \item If $i_1\dots i_k\not\in \mathcal{S}^0_{\mathcal{A}_d}$ and
    \begin{enumerate}
        \item $i_2\dots i_{k}\in\mathcal{S}^0_{\mathcal{A}_d}$, then $U_{i_1\dots i_k}$ is the preimage of $U_{i_2\dots i_k}$ attached to $U_{i_1\dots i_{k-3}}$ on the arc $(z^+_{i_1\dots i_{k-3}j},z^-_{i_1\dots i_{k-3}i_1})$, where $j=i_1-1\pmod{d}$.
        \item $i_2\dots i_{k}\not\in\mathcal{S}^0_{\mathcal{A}_d}$, then $U_{i_1\dots i_k}$ is the preimage of $U_{i_2\dots i_k}$ attached to $U_{i_1\dots i_{k-1}}$.
    \end{enumerate}
\end{enumerate}

As the only drops that intersect the forward orbit of $c_+$ are those with labels in 
\begin{equation}
\label{eq:setS2}
\mathcal{S}_{\mathcal{A}_d}^2=\left\{i_1\dots i_k\in\Xi_{\mathcal{A}_d}\colon i_j\in\{\underline{d},d\},\  j=1,\dots k\right\}, 
\end{equation}
then, 
\[
	B|_{\overline{U}_{i_1i_2\dots i_k}}\colon \overline{U}_{i_1\dots i_k}\to \overline{U}_{i_2\dots i_k},
\]
is a biholomorphism for $i_1\dots i_k \in \Xi_{\mathcal{A}_d}\setminus \mathcal{S}^2_{\mathcal{A}_d}$.

\begin{figure}
\begin{center}
	\includegraphics[width=1\linewidth]{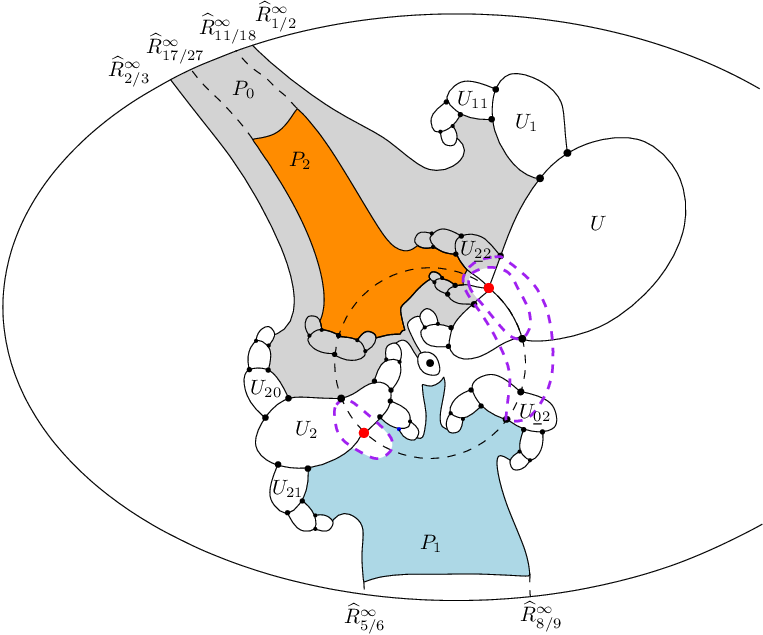}
	\caption{Super-attracting endomorphism scenario for $d=2$: A sketch of the regions $P_0$ (in gray), $P_1$ (in blue), and $P_2$ (in orange). The figure also depicts the super-attracting cycle $\{c_+,z_2^+\}$ (in red) and the boundaries of the removed disks (in purple).}
	\label{fig:EndoSuper}
\end{center}
\end{figure}

\vspace{0.5cm}
\noindent\textbf{Puzzle Pieces Construction. }
The construction of the curve $\Gamma$ follows the procedure outlined for the super-attracting homeomorphism case, with a single necessary modification in the initial step.

In the homeomorphism case, the initial segment from $0$ to $a$ was chosen to pass through the unique critical point $c$. In the current endomorphism case, two critical points $c_+$ and $c_-$ exist on $\mathbb{S}^1$. Due to their symmetry with respect to the ray of angle $2\pi\alpha$ (the line connecting $0$ and $a$), we adjust this initial segment to pass through the midpoint of the arc $(c_-, c_+)$, given by $z = e^{2\pi i \alpha}$.

Beyond this adjustment, the subsequent steps for constructing $\Gamma$ and defining the limbs proceed identically to the earlier case.

The construction generally follows the logic of the super-attracting homeomorphism case. However, a modification is required in the definition of the initial piece $P_0$ to ensure topological separation from its preimage.

Since $\{c_+, z_2^+\}$ forms a super-attracting cycle, we can define stable neighborhoods within their immediate basins. Let $\phi$ denote the linearized K\oe nigs coordinate near $c_+$. For sufficiently small $\varepsilon > 0$, we define the open set $D_{c_+} = \phi^{-1}(D(0,\varepsilon)) \subset \mathcal{A}^*(c_+)$ such that $c_- \in \partial D_{c_+}$. Similarly, let $D_{z_2^+}$ be the unique preimage of $D_{c_+}$ contained in $\mathcal{A}^*(z_2^+)$, satisfying $z_2^- \in \partial D_{z_2^+}$. These domains satisfy the inclusion properties $\overline{D}_{c_+} \subset B^{-2}(\overline{D}_{c_+})$ and $\overline{D}_{z_2^+} \subset B^{-2}(\overline{D}_{z_2^+})$.

We define the puzzle piece $P_0$ by excising these super-attracting neighborhoods from the base piece $P$ (defined in Section \ref{def:base_puzzle_P}):
\[
    P_0 := P \setminus \left(\overline{D}_{c_+} \cup \overline{D}_{z_2^+}\right).
\]

Next, we define the puzzle piece $P_1$ as the connected component of $B^{-1}(P_0)$ contained in the sector $W$ (see Figure \ref{fig:EndoSuper}).

Crucially, the removal of the topological disks $\overline{D}_{c_+}$ and $\overline{D}_{z_2^+}$ ensures that the closures are disjoint: $\overline{P}_0 \cap \overline{P}_1 = \emptyset$. The construction of $P_2$ then proceeds identically to the method described on page \pageref{def:P2_first_construction}.

\paragraph{Attracting case}

Consider a parameter $(r,\alpha) \in T_{1/2}$ in the endomorphism region ($1 < r < 2d+1$) possessing an attracting cycle $\{\xi_0, \xi_1\} \subset \mathbb{S}^1$. We focus on the scenario where both free critical points $c_+$ and $c_-$ lie within the immediate basin of attraction of the same periodic point, denoted $\mathcal{A}^*(\xi_0)$.

Depending on the specific location of $\xi_0$ relative to the critical and co-critical points, we classify the parameter into two distinct types:

\begin{itemize}
    \item Inner Case: The attracting point $\xi_0$ lies in the interval $(c_-, c_+)$.
    \item Outer Case: The attracting point $\xi_0$ lies in the union of intervals $(c_+', c_-) \cup (c_+, c_-')$.
\end{itemize}

This distinction is fundamental to the subsequent analysis because it dictates the pisition of the preimages of the drop $U_d$.
\begin{itemize}
    \item In the \textit{Inner Case}, certain preimages of $\partial U_d$ connect the boundary of the drop $\partial U$ directly to the unit circle $\mathbb{S}^1$.
    \item In the \textit{Outer Case}, preimages are either attached solely to $\partial U$ or solely to $\mathbb{S}^1$, remaining topologically distinct.
\end{itemize}
Due to these differences, we construct the puzzle pieces separately for each scenario.

\paragraph{Inner attracting} We assume that $\xi_0\in (c_-,c_+)$ and $c_+,c_-\in \mathcal{A}^*(\xi_0)$. Therefore, $$(v_+,v_-)\subset (z_d^-,z_d^+).$$

\vspace{0.5cm}
\noindent\textbf{Labeling drops of level $k\geq 2$}

Since the intervals $(c_+', c_-)$ and $(c_+, c_-')$ map injectively to $(v_+, v_-)$, there exist exactly two preimages of the drop $U_d$ whose boundaries intersect the unit circle $\mathbb{S}^1$. We distinguish them based on their position relative to the line $L = \{t e^{i\alpha} \colon t \in \mathbb{R}\}$ connecting $0$ to $a$:

\begin{itemize}
    \item Right Half-Plane: The preimage is labeled $U_{\underline{0}d}$.
    \item Left Half-Plane: The preimage is labeled $U_{\underline{d}d}$.
\end{itemize}

A key topological feature of the Inner Case is that the boundaries of these two drops, $\partial U_{\underline{0}d}$ and $\partial U_{\underline{d}d}$, intersect both the unit circle $\mathbb{S}^1$ and the boundary of the drop $\partial U \setminus \mathbb{S}^1$.

The remaining preimages of $U_d$, denoted by $U_{\underline{1}d}, \dots, U_{\underline{d-1}d}$, do not touch the unit circle and are attached to $\partial U$.

\vspace{0.5cm}
\noindent\textbf{Puzzle Pieces Construction.} 
To prove Lemma \ref{lemma:final-puzzlepiecesparticular} in this scenario, we establish the definition of the initial puzzle piece $P_0$.

Unlike the super-attracting case, the specific distribution of drops in the inner attracting scenario naturally ensures the separation of the puzzle pieces (see Figure \ref{fig:EndoInner}). Consequently, no removal of topological disks is required. We simply define $P_0$ as the base puzzle piece $P$ introduced in Section \ref{def:base_puzzle_P}:
\[
    P_0 := P.
\]
The subsequent piece $P_2$ is then defined following the standard procedure established in previous sections.

\begin{figure}[htbp]
\begin{center}
	\includegraphics[width=1\linewidth]{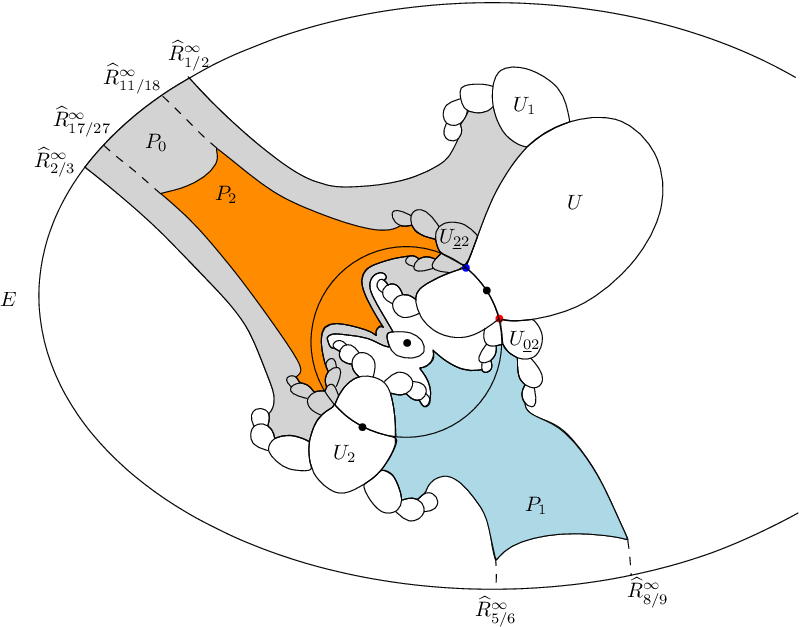}
	\caption{Inner attracting endomorphism scenario for $d=2$: A sketch of the regions $P_0$ (in gray), $P_1$ (in blue), and $P_2$ (in orange). In this case, no disks are removed to define $P_0$, as the natural distribution already ensures that $\overline{P_0} \cap \overline{P_1} = \emptyset$. As before, $P_2 = B^{-1}(P_1)$ is strictly contained in $P_0$.}
	\label{fig:EndoInner}
\end{center}
\end{figure}

\paragraph{Outer attracting}
We now assume that $\xi_0\in (c_+',c_-)\cup (c_+,c_-')$ and $c_+,c_-\in\mathcal{A}^*(\xi_0 )$. Therefore, $$(v_+,v_-)\cap (z_d^-,z_d^+)=\emptyset.$$

\noindent\textbf{Labeling drops of level $k\geq 2$.}
The position of the preimages in the Outer Case closely mirrors that of the super-attracting endomorphism scenario.

Since the arc $[c_-', c_+']$ maps injectively to $[v_-, v_+]$ and orientation is preserved on $[c_+, c_-]$, there exists a unique preimage of the arc $[z_d^-, z_d^+]$ on the unit circle. Consequently, we have the following configuration (see Figure \ref{fig:EndoOuter-At-Step02}):
\begin{itemize}
    \item A unique preimage of $U_d$ is attached to $\mathbb{S}^1$.
    \item The remaining $d$ preimages of $U_d$ are attached solely to $\partial U$.
\end{itemize}

As in previous cases, the label for the unique preimage attached to $\mathbb{S}^1$ is determined by its position relative to the line $L$ connecting $0$ and $a$ (i.e., whether it lies in the left or right open half-plane).

The labeling of the preimages depends on the location of the attracting periodic point $\xi_0$ within the outer intervals:

\begin{enumerate}
    \item \textit{Case $\xi_0 \in (c_+, c_-')$:}
    The unique preimage of $U_d$ attached to $\mathbb{S}^1$ lies in the right half-plane and is labeled $U_{\underline{0}d}$. The remaining preimages attached to $\partial U$ are labeled $U_{id}$ for $i = \underline{1}, \dots, \underline{d}$ in increasing order.

    \item \textit{Case $\xi_0 \in (c_+', c_-)$:}
    The unique preimage of $U_d$ attached to $\mathbb{S}^1$ lies in the left half-plane and is labeled $U_{\underline{d}d}$. The remaining preimages attached to $\partial U$ are labeled $U_{id}$ for $i = \underline{0}, \dots, \underline{d-1}$ in increasing order.
\end{enumerate}

\begin{figure}
\centering
	\includegraphics[width=1\linewidth]{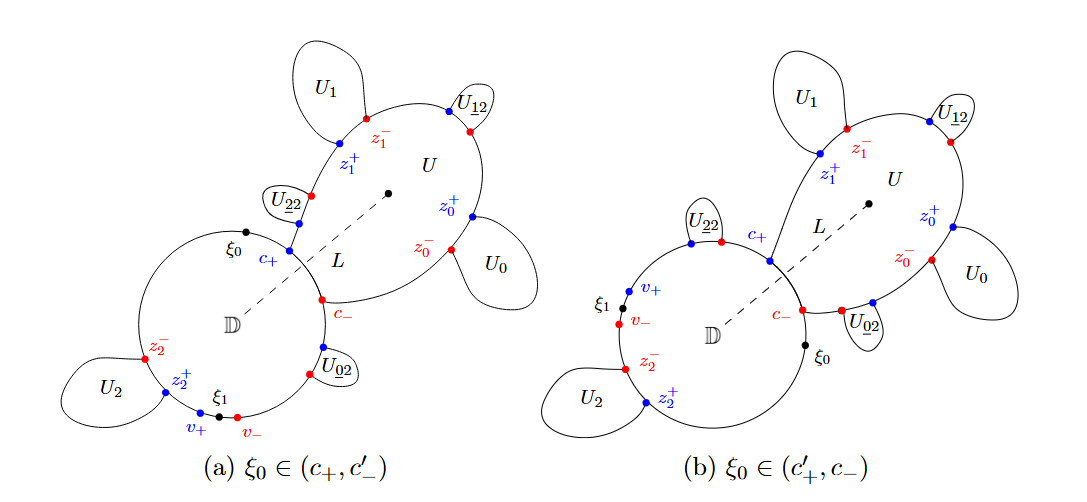}
 \caption{Outer attracting endomorphism scenario for $d=2$: The two possible configurations for the preimages of $U_2$, depending on the location of the attracting point $\xi_0$. (a) If $\xi_0 \in (c_+, c_-')$, the unique preimage attached to $\mathbb{S}^1$ is $U_{\underline{0}2}$. (b) If $\xi_0 \in (c_+', c_-)$, the unique preimage attached to $\mathbb{S}^1$ is $U_{\underline{2}2}$.}
 \label{fig:EndoOuter-At-Step02}
\end{figure}

Since the two configurations in the Outer Case are analogous, we restrict our detailed analysis to the scenario depicted in Figure \ref{fig:EndoOuter-At-Step02}. The construction for the alternative case proceeds symmetrically.

\vspace{0.5cm}
\noindent\textbf{Puzzle Pieces Construction.} 
For this scenario, we must redefine the puzzle piece $P_0$, as in previous cases.
However, the procedure in this case is similar to the attracting homeomorphism case: since $\{\xi_0,\xi_1\}$ is an attracting cycle, there is $\epsilon>0$ such that $c\in \partial D_{0} =\partial \phi^{-1}\left(D(0,\epsilon )\right)$, where $\phi$ are the K\oe nigs coordinates, with $\overline{D}_{0}\subset \mathcal{A}^*(\xi_0)$.
Note that $D_{0}$ does not intersect $\overline{U}$. Let $D_{2}$ be the second preimage of $D_{0}$ under $B$ such that $\overline{D}_{0}\subset \overline{D}_2\subset \mathcal{A}^*(\xi_0)$.
Define the puzzle piece $P_0$ as the open set
\[
    P_0 := P\setminus \overline{D}_2,
\]
where $P$ is the base puzzle piece defined in Section \ref{def:base_puzzle_P}. 

\begin{figure}
\begin{center}
	\includegraphics[width=1\textwidth]{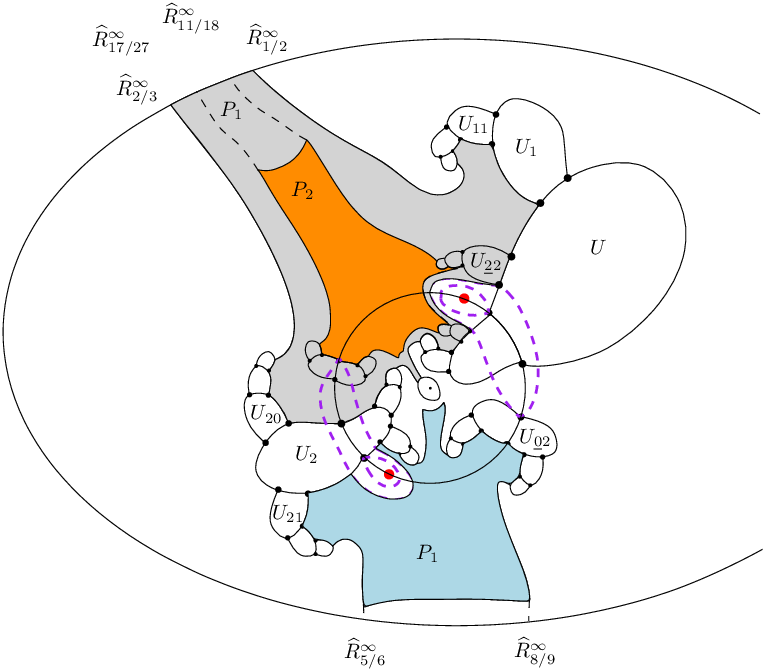}
	\caption{Outer attracting endomorphism scenario for $d=2$: A sketch of the regions $P_0$ (in gray), $P_1$ (in blue), and $P_2$ (in orange). The figure also depicts the attracting cycle $\{\xi_0,\xi_1\}$ (in red) and the boundaries of the removed disks (in purple).}
	\label{fig:EndoOuter}
\end{center}
\end{figure}

\subsection{The Diffeomorphism Case}
We examine the region $T_{1/2} \cap \{r > 2d+1\}$, where the restriction $B|_{\mathbb{S}^1}$ acts as a circle diffeomorphism. In this region, the map possesses two free critical points $c_\pm$ that are symmetric with respect to the unit circle ($\mathcal{I}(c_\pm) = c_\mp$). They are located such that $c_- \in \mathbb{D}$ and $c_+ \in \mathbb{C} \setminus \overline{\mathbb{D}}$. Notably, the points $c_-$, $c_+$, the parameter components $0, a, 1/\bar{a}$, and the midpoint $e^{2\pi i\alpha}$ are all collinear.

The circle dynamics exhibit two distinct periodic orbits of period 2:
\begin{itemize}
    \item An attracting 2-cycle, denoted $\{\xi_0, \xi_1\}$.
    \item A repelling 2-cycle, denoted $\{w_0, w_1\}$.
\end{itemize}
The primary objective of this section is to prove that the repelling cycle $\{w_0, w_1\}$ is \emph{bi-accessible}. As defined previously, let $U$ denote the the bounded domain in $\widehat{\mathbb{C}} \setminus \overline{\mathbb{D}}$ mapping to $\mathbb{D}$ and containing $a$. Its boundary $\partial U$ maps $d$-to-one onto $\mathbb{S}^1$. A \emph{drop of level $k$} is any connected component of $B^{-k}(U)$.

\vspace{0.5cm}
\noindent\textbf{Labeling Drops, Roots and Necks.}
We label the connected components of $B^{-1}(U)$ using admissible words from Subsection \ref{sec:rules-admissibility}. In the diffeomorphism region, these preimages are pairwise disjoint. To establish a reference frame, we define:
\begin{itemize}
    \item The reference point $z_{\text{ref}} = e^{2\pi i \alpha} \in \mathbb{S}^1$, which is the intersection of the unit circle and the line segment connecting $0$ to $a$.
    \item The \emph{neck} of $U$, denoted $\tilde{z} \in \partial U$, as the unique point on $\partial U$ intersecting the segment from $0$ to $a$.
    \item The \emph{shaft}, denoted $S$, as the line segment connecting the neck $\tilde{z}$ to the reference point $z_{\text{ref}}$.
\end{itemize}

The set $B^{-1}(z_{\text{ref}})$ contains $2d+1$ points. Since $B|_{\mathbb{S}^1}$ is a degree-one diffeomorphism, exactly one preimage lies on $\mathbb{S}^1$. Since $B|_{\partial U}$ is a degree-$d$ covering, exactly $d$ preimages lie on $\partial U$. We label the corresponding drops $U_i$ as follows:

\begin{enumerate}
    \item \textbf{Drops associated to $\partial U$ ($i = 0, \dots, d-1$):}
    Let $\{z_0, \dots, z_{d-1}\}$ be the $d$ preimages of $z_{\text{ref}}$ located on $\partial U$. We label them along the boundary $\partial U$, starting from the neck $\tilde{z}$. For each $z_i$:
    \begin{itemize}
        \item $U_i$ is the component of $B^{-1}(U)$ associated with $z_i$.
        \item The \emph{root} of $U_i$ is $z_i$.
        \item The \emph{neck} $\tilde{z}_i \in \partial U_i$ is the unique preimage of $\tilde{z}$ on the boundary of $U_i$.
        \item The \emph{shaft} $S_i$ is the component of $B^{-1}(S)$ connecting $\tilde{z}_i$ to $z_i$.
    \end{itemize}

    \item \textbf{Drop associated to $\mathbb{S}^1$ ($i = d$):}
    Let $z_d$ be the unique preimage of $z_{\text{ref}}$ located on $\mathbb{S}^1$.
    \begin{itemize}
        \item $U_d$ is the component of $B^{-1}(U)$ associated with $z_d$.
        \item The \emph{root} of $U_d$ is $z_d \in \mathbb{S}^1$.
        \item The \emph{neck} $\tilde{z}_d \in \partial U_d$ is the unique preimage of $\tilde{z}$ on $\partial U_d$.
        \item The \emph{shaft} $S_d$ is the component of $B^{-1}(S)$ connecting $\tilde{z}_d$ to $z_d$.
    \end{itemize}
\end{enumerate}

Figure \ref{fig:DifeoLevel1} illustrates this configuration for $d=2$.

\begin{figure}
\begin{center}
	\includegraphics[scale=0.9]{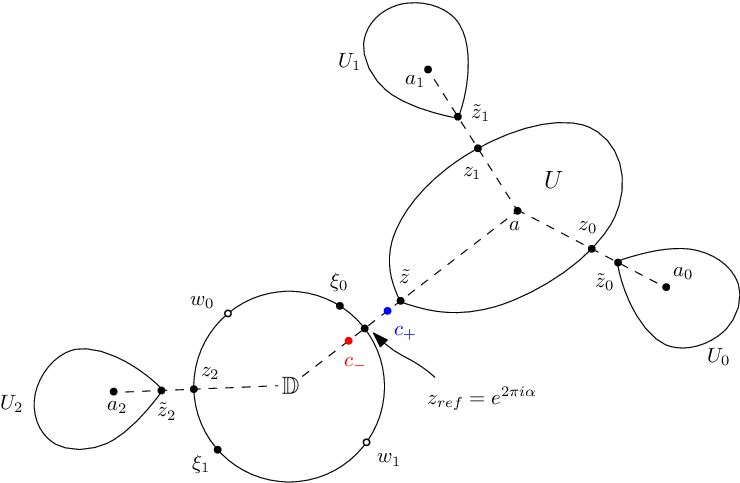}
	\caption{Sketch of the drops, roots and necks of level 1 for the case $d=2$. The preimages of the level 0 drop $U$ are shown: $U_0$ and $U_1$ and their roots $z_0$, $z_1$ and necks $\tilde{z}_0$ and $\tilde{z}_1$, while $U_2$ is the unique preimage of $z_ref=e^{2\pi i\alpha}$ on the circle,  associated with the root $z_2$ on $\mathbb{S}^1$, with neck $\tilde{z}_2$.}
    \label{fig:DifeoLevel1}
\end{center}
\end{figure}

\noindent\textbf{Drops of level $2$}

The labeling of the preimages of $U_d$ (denoted $U_{id}$ with $i=\underline{0},\dots , \underline{d}$) mirrors the procedure in the Attracting Homeomorphism case. We assign labels based on the position of the reference angle $\arg(e^{2\pi i \alpha})$ relative to the attracting fixed point $\xi_0$:

\begin{enumerate}
    \item \textbf{Case $\arg (e^{2\pi i \alpha}) \leq \arg (\xi_0)$:}
    The unique preimage associated to $\mathbb{S}^1$ lies to the right of the line segment from $0$ to $a$ and is labeled $U_{\underline{0}d}$. The remaining components attached to $\partial U$ are labeled $U_{id}$ for $i=\underline{1},\dots, \underline{d}$ in increasing order relative to the orientation of $\partial U$ starting from the neck $\tilde{z}$.

    \item \textbf{Case $\arg (e^{2\pi i\alpha}) > \arg (\xi_0)$:}
    The unique preimage assocaited to $\mathbb{S}^1$ is labeled $U_{\underline{d}d}$. The remaining components associated to $\partial U$ are labeled $U_{id}$ for $i=\underline{0},\dots, \underline{d-1}$ in increasing order relative to the orientation of $\partial U$ starting from $\tilde{z}$.
\end{enumerate}

These configurations are illustrated in Figure \ref{fig:Difeo-At-Step02}.

\begin{figure}
\centering
	\includegraphics[scale=0.75]{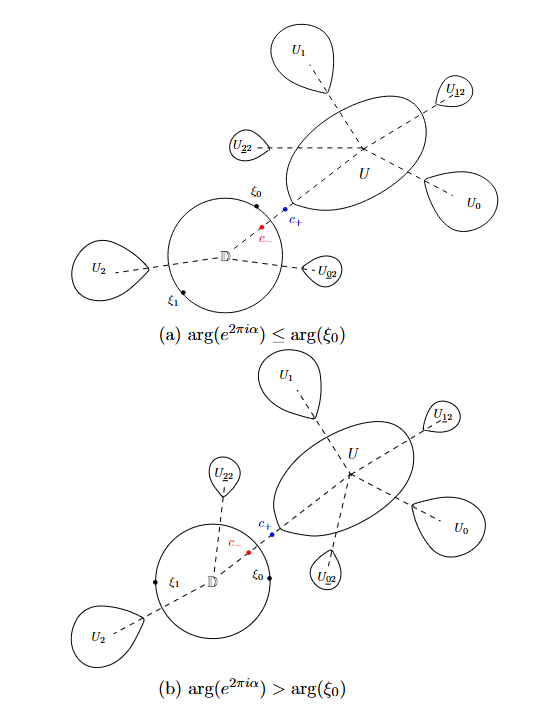}
 \caption{Diffeomorphism scenario for $d=2$: The two possible configurations for the preimages of $U_2$, depending on the location of the attracting point $\xi_0$. (a) If $\arg (e^{2\pi i \alpha}) \leq \arg (\xi_0)$, the unique preimage associated to $\mathbb{S}^1$ is $U_{\underline{0}2}$. (b) If $\arg (e^{2\pi i\alpha}) > \arg (\xi_0)$, the unique preimage associated to $\mathbb{S}^1$ is $U_{\underline{2}2}$.}
 \label{fig:Difeo-At-Step02}
\end{figure}
\vspace{0.5cm}
\noindent\textbf{Drops of level $k\geq 3$}

For higher levels ($k \geq 3$), we adopt the labeling scheme from the attracting homeomorphism case. We first define the subset of admissible words $\mathcal{S}^0_{\mathcal{A}_d}$ consisting of drops attached to the circle dynamics:
\begin{equation}
    \mathcal{S}^0_{\mathcal{A}_d} = \left\{ i_1 \dots i_k \in \Xi_{\mathcal{A}_d} \colon i_j \in \{\underline{0}, d\}, \ j=1, \dots, k \right\}.
\end{equation}

Given the drops up to level $k-1$, the drop $U_{i_1 \dots i_k}$ is defined by the following recursive rules:

\begin{enumerate}
    \item If the word $i_1 \dots i_k \in \mathcal{S}^0_{\mathcal{A}_d}$, then $U_{i_1 \dots i_k}$ is the unique preimage of $U_{i_2 \dots i_k}$ associated to $\mathbb{S}^1$.

    \item If the word $i_1 \dots i_k \notin \mathcal{S}^0_{\mathcal{A}_d}$, the association depends on the shifted word $i_2 \dots i_k$:
    \begin{itemize}
        \item If $i_2 \dots i_k \in \mathcal{S}^0_{\mathcal{A}_d}$, then $U_{i_1 \dots i_k}$ is the preimage of $U_{i_2 \dots i_k}$ associated to the drop $U_{i_1 \dots i_{k-3}}$ along the arc $(z_{i_1 \dots i_{k-3}j}, z_{i_1 \dots i_{k-3}i_1})$, where $j \equiv i_1-1 \pmod{d}$.
        \item If $i_2 \dots i_k \notin \mathcal{S}^0_{\mathcal{A}_d}$, then $U_{i_1 \dots i_k}$ is the preimage of $U_{i_2 \dots i_k}$ associated to the drop $U_{i_1 \dots i_{k-1}}$.
    \end{itemize}
\end{enumerate}

As with previous levels, the map $B$ acts on these drops as a right-hand shift on the symbolic labels.
\vspace{0.5cm}

\noindent\textbf{General Shafts ($k \geq 2$)}

For levels $k \ge 2$, we define the shafts recursively based on the construction at level 1. Let $w = i_1 i_2 \dots i_k \in \Xi_{\mathcal{A}_d}$ be an admissible word, and assume the shaft $S_{i_2 \dots i_k}$ for the shifted word has already been defined.

We define the \emph{shaft} $S_{i_1 i_2 \dots i_k}$ as the unique connected component of $B^{-1}(S_{i_2 \dots i_k})$ attached to the drop $U_{i_1 i_2 \dots i_k}$. It satisfies the following properties:

\begin{itemize}
    \item It connects the neck $\tilde{z}_{i_1 \dots i_k} \in \partial U_{i_1 \dots i_k}$ to the root $z_{i_1 \dots i_k}$.
    \item The root $z_{i_1 \dots i_k}$ lies on the boundary of the preceding drop $U_{i_1 \dots i_{k-1}}$ (or on $\mathbb{S}^1$ if the drop is attached directly to the circle).
    \item The map $B$ maps the shaft $S_{i_1 \dots i_k}$ bijectively onto the shaft $S_{i_2 \dots i_k}$.
\end{itemize}

\vspace{0.5cm}
\noindent\textbf{Limbs.} 
In the diffeomorphism region, the drops are topologically disjoint. To ensure the connectedness of the limbs, we extend their definition to explicitly include the connecting shafts $S_i$.

Let $j \in \{0, d-1\}$. We denote by $j^k$ the word of length $k$ consisting entirely of the symbol $j$, and by $\overline{j}$ the infinite sequence of $j$'s. We define the \emph{limbs} starting at $U$ as:
\begin{equation*}
    L_{\overline{j}} = \overline{U \cup \bigcup_{k\geq 1} U_{j^k}} \cup \bigcup_{k\geq 1} S_{j^k}.
\end{equation*}

We say that a limb $L_{\overline{j}}$ \emph{starts at} $U$, meaning $U$ is the drop of the lowest level within the set.

We recursively define the preimages of these limbs. For any index $i \in \{0, \dots, d\}$, let $L_{i \cdot \overline{j}}$ denote the connected component of $B^{-1}(L_{\overline{j}})$ contained in $\mathbb{C} \setminus \mathbb{D}$ that starts at the drop $U_i$.

For any $k \in \mathbb{N}$ and any admissible word $w = i_1 \dots i_k \in \Xi_{\mathcal{A}_d}$, we define the limb $L_{w \cdot \overline{j}}$ as the unique connected component of $B^{-1}(L_{i_2 \dots i_k \cdot \overline{j}})$ that starts at the drop $U_{w}$.

\begin{lemma}
\label{lemma:diam_drops_segments}
	Let $j \in \{0, d-1\}$. Consider the drop of level $k$ given by $U_{j^k} \subset L_{\overline{j}}$ and the shaft $S_{j^k}$. Then:
	\[
		\lim_{k \to \infty} \operatorname{diam}(\overline{U}_{j^k}) = 0 \quad \text{and} \quad \lim_{k \to \infty} \operatorname{length}(S_{j^k}) = 0.
	\]
\end{lemma}
\begin{proof}
	The convergence of the drops $\overline{U}_{j^k}$ follows from the same argument as in the homeomorphism case (Lemma \ref{lemma:diam0}), as the inverse branch $\phi _j$ is a strict contraction on the compact set $\overline{U} \cup \bigcup \overline{U}_i$. 
    
    Since the shaft $S_{j^k}$ connects the root of $U_{j^k}$ to its neck, and $S_{j^k} = \phi_j^{k-1}(S_j)$, the contraction of the inverse branch $\phi_j$ also implies that the length of these segments converges geometrically to zero as $k \to \infty$.
\end{proof}

\begin{lemma}
\label{lemma:intersection_is_point_diffeo}
Let $j \in \{0, d-1\}$. Then, the intersection
\begin{equation}
\label{eq:FixedPointInBoundaryDiffeo}
	\bigcap_{k=1}^\infty L_{j^k\cdot \overline{j}}
\end{equation}
is a single point, denoted as $\zeta_j$, which is fixed under $B$ and lies on $\mathbb{C} \setminus \overline{\mathbb{D}}$.
\end{lemma}

\begin{proof}
    By definition, the limb segment $L_{j^k \cdot \overline{j}}$ is the closure of the union of nested drops and their shafts starting from level $k$:
    \[
    L_{j^k \cdot \overline{j}} = \overline{\bigcup_{l=k}^{\infty} \left( U_{j^l} \cup S_{j^l} \right)}.
    \]
    These form a nested sequence of compact, connected sets: $L_{j^{k+1}\cdot \overline{j}} \subset L_{j^k\cdot \overline{j}}$ for all $k \ge 1$.
    
     The diameter of the limb segment is bounded by the sum of the diameters of its components.
    \[
        \operatorname{diam}(L_{j^k\cdot \overline{j}}) \le \sum_{l=k}^{\infty} \left( \operatorname{diam}(\overline{U}_{j^l}) + \operatorname{diam}(S_{j^l}) \right).
    \]
    From Lemma \ref{lemma:diam_drops_segments}, both terms geometrically converge to zero. Therefore, the tail of the series converges to zero, implying:
    \[
        \lim_{k \to \infty} \operatorname{diam}(L_{j^k\cdot \overline{j}}) = 0.
    \]
    By Cantor's Intersection Theorem, the intersection of this nested sequence of non-empty compact sets is a single point $\{\zeta_j\}$.
    
    Since $B$ maps the limb $L_{j^{k+1}\cdot \overline{j}}$ onto $L_{j^k\cdot \overline{j}}$, the intersection point must be fixed under $B$, i.e., $B(\zeta_j) = \zeta_j$. 
    Finally, since the entire structure lies in $\mathbb{C} \setminus \overline{\mathbb{D}}$, the fixed point $\zeta_j$ is in $\mathbb{C} \setminus \overline{\mathbb{D}}$.
\end{proof}

With these modified definitions ensuring that the limbs are shrinking under $B^{-n}$, the remainder of the construction and analysis of the limbs proceeds in the exact same way as in the super-attracting homeomorphism case.

\vspace{0.5cm}
\noindent\textbf{Puzzle Pieces Construction.} 
We keep the established notation for B\"ottcher rays and the sector $W$. The construction of the initial puzzle piece $P_0$ follows the logic of the Attracting Homeomorphism case, utilizing the attracting nature of the cycle $\{\xi_0, \xi_1\}$.

Let $\phi$ denote the linearized K\oe nigs coordinates near $\xi_0$. Due to the symmetry of $B$ under the involution $\mathcal{I}(z)$, we can find an $\epsilon > 0$ such that the domain $D_0 = \phi^{-1}(D(0, \epsilon)) \subset \mathcal{A}^*(\xi_0)$ satisfies the condition that both critical points lie on its boundary ($c_+, c_- \in \partial D_0$). Note that $D_0$ is disjoint from $\overline{U}$.

Let $D_2$ be the connected component of $B^{-2}(D_0)$ satisfying the inclusion $\overline{D}_0 \subset \overline{D}_2 \subset \mathcal{A}^*(\xi_0)$. Crucially, $\overline{D}_2$ does not intersect the open drops $U_{\underline{0}d}$ and $U_{\underline{d}d}$.

We define the puzzle piece $P_0$ by removing this trapping region from the base piece $P$ (defined in Section \ref{def:base_puzzle_P}):
\[
    P_0 \colon = P \setminus \overline{D}_2.
\]

\begin{figure}
\begin{center}
	\includegraphics[width=1\textwidth]{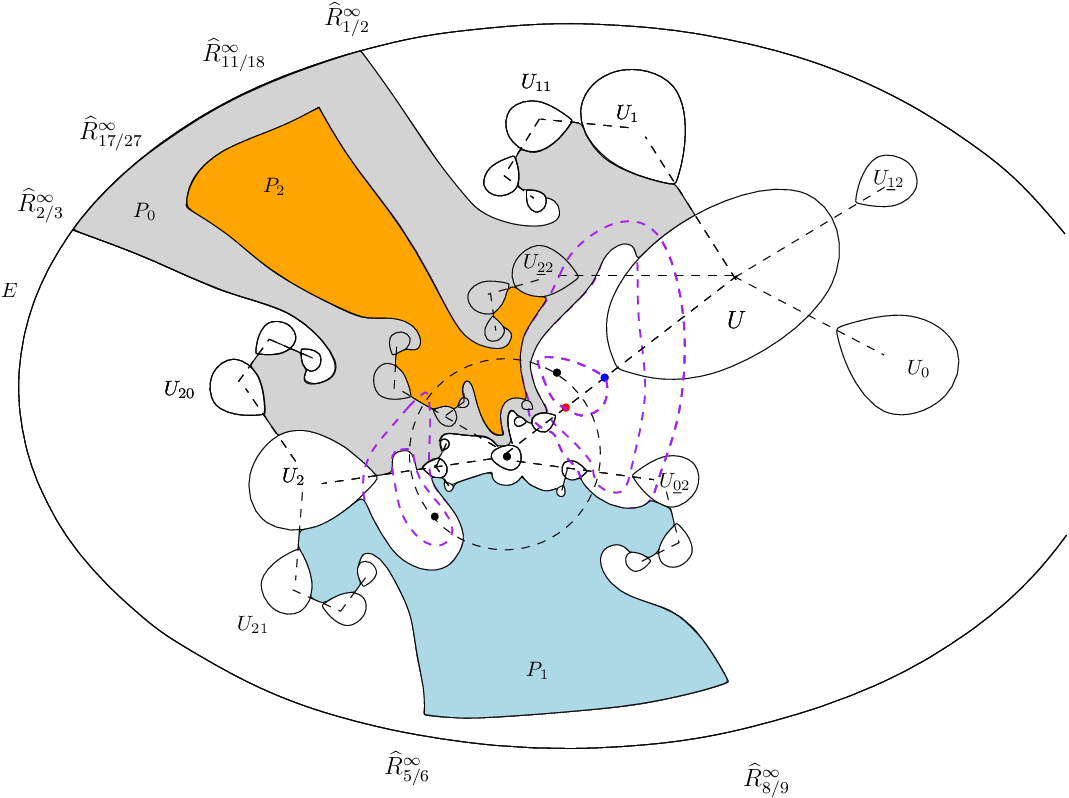}
	\caption{Diffeomorphism scenario for $d=2$: A sketch of the regions $P_0$ (in gray), $P_1$ (in blue), and $P_2$ (in orange). The figure also depicts the attracting cycle $\{\xi_0,\xi_1\}$ (in black) and the boundaries of the removed disks (in purple). The critical points $c_+$ and $c_-$ are colored in blue and red, respectively.}
	\label{fig:Difeo}
\end{center}
\end{figure}

\subsection{General Rotation Number}
\label{subsec:Generalizations}
We have established the 
puzzle pieces construction for 
the period 2 cycle with rotation number $\rho(X) = 1/2$. Specifically, for the case $d=2$, analyzing the map $m_3$ required constructing an interval bounded by preimages of $0$ and $1/2$ that compactly covers itself under $q=2$ iterations and isolates a single point of the cycle $\{5/8, 7/8\}$.

For the map $m_{d+1}$ and by Proposition~\ref{prop:boundary_dynamics_and_fixed_points}, we now generalize this construction for any given rational rotation number.  Our objective is to prove the existence of an interval with the following properties for any given $\rho = p/q$:
\begin{itemize}
    \item Its endpoints are preimages of two consecutive fixed points.
    \item It compactly covers itself under $q$ iterations.
    \item It contains exactly one point of the cycle with rotation number $p/q$.
\end{itemize}

We rely on Theorem \ref{thm:Goldberg}, which guarantees that for every irreducible fraction $0 < p/q < 1$, there exists a cycle $X = \{t_1, \dots, t_q\}$ under $m_{d+1}(t) = (d+1)t \pmod{\mathbb{Z}}$ with rotation number $\varrho(X) = p/q$ located within the interval:
\[
    X \subset \left[\frac{d}{d+1}, 1\right] \equiv \left[\frac{d}{d+1}, 0\right] \pmod{\mathbb{Z}}.
\]

\begin{proposition}
\label{prop:gen_interval}
Fix $d\geq 1$. Let $0<p/q<1$ be a fraction in lowest terms and let $X=\{t_1,\ldots,t_q\}\subset [(d-1)/d, 1)$ be a $q$-cycle under $m_{d+1}$ with rotation number $\rho(X)=p/q$. Then, there exist $a,b\in [(d-1)/d, 1)$ such that
\begin{enumerate}
    \item $m_{d+1}^q(a)=(d-1)/d$, $m_{d+1}^{q+1}(b)=0\equiv 1 \pmod{\mathbb{Z}}$,
    \item $t_1\in (a,b)$, $t_j\notin (a,b)$ for $j\neq 1$, and
    \item $[a,b]\subset m_{d+1}^q([a,b])$.
\end{enumerate}
\end{proposition}
\begin{proof}
Fix $d\geq 1$ and set $I:=[(d-1)/d,1)\subset \mathbb{R}$. For simplicity, consider the action of $m_{d+1}$ on the  interval $[0,1)$, that is, $m_{d+1}:[0,1)\to [0,1)$ given by $m_{d+1}(t)=(d+1)t-\lfloor (d+1)t \rfloor$. In particular, the restriction $m|_I:I\to [0,1)$ is given by
\[
m(t) = \begin{cases}
    (d+1)t-(d-1) & \text{if } \frac{d-1}{d}\leq t< \frac{d}{d+1}, \\
    (d+1)t-d & \text{if } \frac{d}{d+1}\leq t< 1.\end{cases}
\]
Each endpoint of $I$ has exactly two preimages under $m$: the endpoint itself and a second preimage lying in the interior of $I$. The non-fixed preimages are given by
\[a'=\frac{d^2+d-1}{d(d+1)}\longmapsto m(a')=\frac{d-1}{d}\qquad\text{and}\qquad b'=\frac{d}{d+1}\longmapsto m(b')=0 \equiv 1\mod \mathbb{Z},\]
and where $(d-1)/d<b'\leq a'<1$. These values generate a partition of $I$ into three subintervals
\[J_0=[(d-1)/d, b'),\qquad J_*=[b',a'),\qquad J_1=[a',1)\]
so that $I=J_0\sqcup J_*\sqcup J_1$. Both $J_0$ and $J_1$ are mapped under $m$ to $I$ while $m(J_*)=[0,(d-1)/d)$. See Figure~\ref{fig:m3}.

\begin{figure}
\centering
     \includegraphics[width=0.7\linewidth]{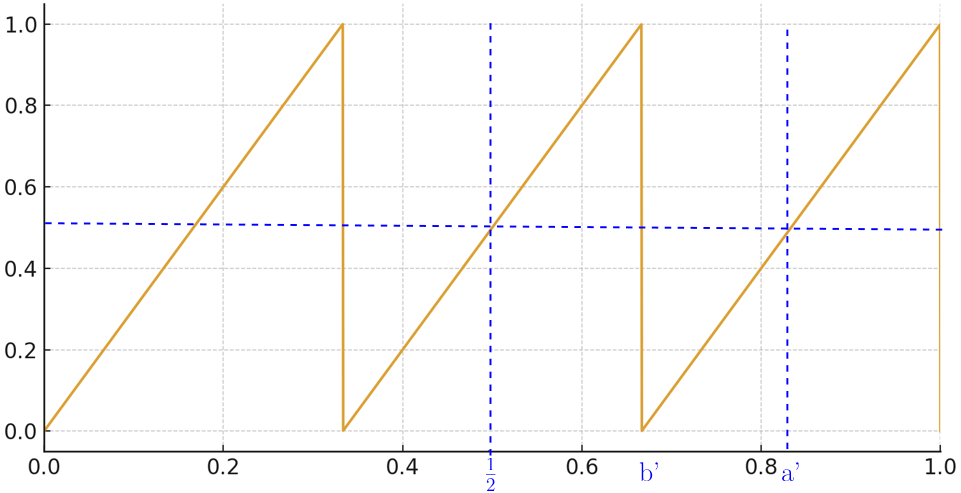}
    \caption{Plot of $m_{d+1}$ for the case $d=2$. The interval $I=[1/2,1)$ is partitioned into three subintervals, with  $J_0=[1/2,b']$ and $J_1=[a',1]$ mapping into $I$, while $J_*=[b',a')$ maps to $[0,1/2)$.}
    \label{fig:m3}
\end{figure}

Since $m|_I$ is an expansive endomorphism, it follows from standard arguments in the theory of interval maps (see for example \cite{Brin2002}), the existence of a positively invariant set,
\[\Lambda\subset J_0\sqcup J_1\subset I,\qquad\text{where}\qquad\Lambda =\{t\in I~:~ m^k(t)\in J_0\sqcup J_1, \forall k\in \mathbb{N}_0\}.\]

Furthermore, there exists a topological conjugacy between the dynamical systems $(\Lambda, m)$ and $(\Sigma_2,\sigma)$, where $\Sigma_2=\{0,1\}^\mathbb{N}$ and $\sigma:\Sigma_2\to \Sigma_2$ is the right-shift mapping. The conjugacy is defined by the homeomorphism (also known as the \emph{itinerary map})
\[S:\Lambda\to \Sigma_2,\qquad S(t)=s_1s_2s_3\ldots,\]
if and only if $m^k(t)\in J_{s_k}$ for all $k\geq 1$.

We now proceed to show the existence of an interval $[a,b]\subset I$ with the required properties. Clearly, $X\subset \Lambda$, so each point in $X$ has a well-defined itinerary. The hypothesis on $p/q$ implies that $q\geq 2$, and hence, $X\cap J_0\neq \emptyset$ and $X\cap J_1\neq \emptyset$, for otherwise, if the $q$-cycle is contained in a single interval $J_s$ for some  $s\in \{0,1\}$, then every point in $X$ has the same itinerary $0000\ldots$ or $1111\ldots$, both of them fixed points of $\sigma$. The conjugacy implies that $X$ reduces to a single fixed point, a contradiction.

Relabeling the points in $X$, let $t_1$ be the smallest value in $X\cap J_0$, so that
\[\frac{d-1}{d}<t_1<t_j<1\]
for all $j=2,\ldots,q$. Let $S(t_1)=\overline{s_1 s_2\ldots s_q}\in \Sigma_2$ be the $q$-periodic itinerary of $t_1$ and observe that $s_1=0$. Define the set

\begin{equation}
\begin{aligned}
J_{s_1\ldots s_q s_1} &:= \{x\in I~:~x\in J_{s_1}, m(x)\in J_{s_2},\ldots,m^{q-1}(x)\in J_{s_q}, m^q(x)\in J_{s_1}\} \\
&= J_{s_1} \cap m^{-1}(J_{s_2})\ldots m^{-q+1}(J_{s_1})\cap m^{-q}(J_{s_1}).
\end{aligned}
\end{equation}
Since $m|_{J_{s_k}}$ is monotone increasing, then $J_{s_1\ldots s_q s_1}$ is a single, non-trivial subinterval of $J_{s_1}$ that contains $t_1$ in its interior. Furtheremore
\[m(J_{s_1\ldots s_q s_1})= J_{s_2}\cap \ldots \cap m^{-q+1}(J_{s_1})=J_{s_2\ldots s_q s_1}\]
since $m(J_{s_1})=I$. Therefore $m^q(J_{s_1\ldots s_q s_1})=J_{s_1}=J_0$. Set $[a,b]=J_{s_1\ldots s_q s_1}$. Then, 

\begin{enumerate}
    \item $m$ restricted to $[a,b]$ is monotone increasing, thus the endpoints of $[a,b]$ are mapped to the endpoints of $J_0$, that is,
    \[m^q(a)=(d-1)/d,\qquad \text{and}\qquad m^q(b)=b',\]
    and therefore $m^{q+1}(b)=0 \equiv 1 \mod \mathbb{Z}$.
    \item If $t_j\in X\cap [a,b]$, then its itinerary $S(t_j)$ starts with the word $s_1\ldots s_q$. By the topological conjugacy, this happens if and only if $j=1$. Thus, $X\cap [a,b]=\{t_1\}$.
    \item Since $[a,b]\subset J_{s_1}=J_0$ and $m^q([a,b])=J_{s_1}$ then $[a,b]\subset m^q([a,b])$.
\end{enumerate}    
\end{proof}

\subsection{Proof of Theorem A}

We now prove Theorem~\ref{thm:biaccessibility}. While the construction presented in this section focused on $\rho=1/2$, the line of reasoning holds for a general rotation number $\rho=p/q$.

\begin{proof}[Proof of Theorem~\ref{thm:biaccessibility}]
    Let $\rho = p/q$. By Proposition~\ref{prop:gen_interval}, there exists an interval $I$ on the circle containing exactly one point $t_1$ of the cycle for $m_{d+1}$ such that $I$ covers itself under $m_{d+1}^q$.
    
    We map the interval structure to the dynamical plane of $B$. We define the region $W$ bounded by the extended Böttcher rays corresponding to the endpoints of $I$. We define the base puzzle piece $P_0 \subset W$ by removing, if needed, small disks around the cycle (if super-attracting) or linearization disks (if attracting).
    
   By Theorem \ref{thm:SteinmetzInverse}, and considering that $P_0$ is a topological disk on the Riemann sphere with no critical points, there exists an analytical inverse of $B^2|_{P_2}\colon P_2\to P_0$, denoted as $\varphi$.  Given that $\overline{\varphi\left(P_0\right)}=\overline{P}_2\subset P_0$, then,  following Theorem \ref{thm:Contraction}, there exists a unique fixed point for $\varphi$ in $P_2$, denoted as $z^*$. Moreover, $z^*$ is also a unique fixed point for $\varphi$ in $P_2$:
\[
	B^2\left(z^*\right)=B^2\left(B^{-2}(z^*)\right) = z^*.
\]

Note that the map $B^2$, when restricted to $P_0\cap \partial\mathcal{A}^*(\infty)$, has a fixed point at $\zeta_{t_1}$, which is the landing point of the ray $\widehat{R}^\infty_{t_1}$ with 
\[
t_1=\frac{d^2+d-1}{d(d+1)}.
\]
Similarly, the restriction of $B^2$ to $P_0\cap\partial\mathcal{A}^*(0)$ has a fixed point at $\zeta_{s_2}$, the landing point of the ray $\widehat{R}^0_{s_2}$, where $s_2\in P_2$. By the uniqueness of $z^*$, it follows that
\[
\zeta_{t_1}=z^*=\zeta_{s_2}.
\]
Therefore, $\zeta_{t_1}$ is a bi-accessible point, and so is $\zeta_{t_2}=B(\zeta_{t_1})$.

Finally, we have obtained a bi-accessible cycle $\{ \zeta_{t_2},\zeta_{t_2}\} $ with rotation number $\varrho=1/2$, as desired.
\end{proof}

\section{Connectedness of the Julia Set}
\label{sec:connectedness}

In this section, we employ the theoretical framework and constructive results developed previously to analyze the connectivity of the Julia set for the generalized Blaschke family. We consider the family of rational maps $B_a \colon \widehat{\mathbb{C}} \to \widehat{\mathbb{C}}$ defined by:
\begin{equation*}
    B_a(z) = z^{d+1} \left( \frac{z-a}{1-\bar{a}z} \right)^d,
\end{equation*}
where $d \ge 1$ is a fixed integer and $a \in \mathbb{C}$.

We utilize the parametrization $a = r e^{2\pi i \alpha}$ with $r \in (0, \infty)$ and $\alpha\in(-\frac{1}{4d},\frac{1}{4d}]$, denoting the map as $B_{r,\alpha}$. As noted in prior discussions regarding Herman rings, the Fatou set of this family may contain doubly connected components for parameters $r=|a| > 2d+1$.  Since the existence of such components prevents the Julia set from being connected, our main result establishes that $\mathcal{J}(B_{r,\alpha})$ is connected if and only if $\mathcal{F}(B_{r,\alpha})$ contains no Herman rings.

The connectedness of the Julia set is intimately linked to the topology of the Fatou components. Recall that the Julia set of a rational map is connected if and only if every Fatou component is simply connected. Given the non-existence of Siegel disks (Theorem~\ref{thm:rotation-domains}) and the established simple connectivity of the basins of $0$ and $\infty$ (Proposition~\ref{thm:super-attracting-basins}), our investigation focuses on the connectedness of periodic attracting or parabolic basins that may arise within the Arnold tongues.

We rely on Theorem~\ref{thm:biaccessibility}, which guarantees the existence of a bi-accessible cycle for adjacent parameters within the tongues. This topological feature acts as an obstacle to the existence of multiply connected basins, thereby ensuring the connectedness of the Julia set.

\subsection{Connectedness Results}

In this subsection, we consider parameters $(r,\alpha)$ with $r \in (1, 2d+1)$, for which both free critical points of $B_{r,\alpha}$ lie in the same Fatou component.

The following lemma establishes that if a point on the unit circle is accessible from both the basin of zero and the basin of infinity, it effectively bridges the dynamical plane.

\begin{lemma}
\label{lemma:path-connected-K}
    Consider $B_{r,\alpha}$ as a generalized Blaschke product, where $r \in (1, 2d+1)$, and both free critical points associated with $B_{r,\alpha}$ reside in the same Fatou component. Suppose there exists a bi-accessible point $\omega_0$ within the Julia set $\mathcal{J}(B_{r,\alpha})$. That is, 
    \[
    \omega_0\in\partial \mathcal{A}^*(0)\cap \partial\mathcal{A}^*(\infty).
    \] 
    Then, the set 
    \[
    K \colon = \overline{\mathcal{A}^*(0)} \cup \overline{\mathcal{A}^*(\infty)} \cup B_{r,\alpha}^{-1}\left( \overline{\mathcal{A}^*(0)} \cup \overline{\mathcal{A}^*(\infty)}\right)
    \] 
    is path connected.
\end{lemma}

\begin{proof}
    Since the map $B_{r,\alpha}$ is hyperbolic and $\mathcal{A}^*\left(\infty \right)$ is simply connected, the boundary $\partial\mathcal{A}^*\left(\infty \right)$ is a Jordan curve. Consequently, for each $z\in\partial\mathcal{A}^*\left(\infty \right)$, we can construct a path, denoted as $\gamma_z^\infty$, within $\overline{\mathcal{A}^*\left(\infty\right)}$ that connects $z$ to $\infty$. In particular, there exists a curve $\gamma_{\omega_0}^\infty $ connecting the bi-accessible point $\omega_0$ to $\infty $ in $\overline{\mathcal{A}^*\left(\infty \right)}$.

    By the symmetry of the map $\mathcal{I}(z)=1/\bar{z}$, the basin $\mathcal{A}^*(0)$ is also simply connected. Thus, for any $z\in\partial\mathcal{A}^*(0)$, there are curves within $\overline{\mathcal{A}^*\left(0\right)}$, denoted as $\gamma_z^0$, that connect $z$ to $0$. For the specific point $\omega_0$, which lies on the common boundary, the union of the curve from $\infty$ to $\omega_0$ and from $\omega_0$ to $0$ forms a path connecting $\infty$ to $0$ passing through the Julia set only at the point $\omega_0$.

    Consider now a component $V$ of $B_{r,\alpha}^{-1}\left(\mathcal{A}^*\left(\infty \right)\right)$. Since $B_{r,\alpha}$ maps $V$ onto $\mathcal{A}^*(\infty)$, any curve $\gamma_\infty$ connecting $\infty$ to $\omega_0$ gives a curve in $V$ connecting a preimage of $\infty$ (which is $\infty$ or the pole) to a preimage of $\omega_0$. Through the concatenation of these curves and the specific path through $\omega_0$, any point in the set $K$ can be connected to $\infty$. Therefore, $K$ is path connected.
\end{proof}

\begin{theorem}
\label{thm:A-atractora-simplemente-conexa}
    Consider $B_{r,\alpha}$ as a generalized Blaschke product, where $r \in (1, 2d+1)$, and both free critical points associated with $B_{r,\alpha}$ reside in the same Fatou component. Assume that the Julia set $\mathcal{J}(B_{r,\alpha})$ contains a bi-accessible point $\omega_0$. If $z_0 \in \mathbb{S}^1$ is an attracting periodic point, then its immediate basin of attraction $\mathcal{A}^*(z_0)$ is simply connected. Consequently, every component of the basin of attraction $\mathcal{A}(z_0)$ is simply connected.
\end{theorem}

\begin{proof}
    We proceed by contradiction. Assume that $\mathcal{A}^*(z_0)$ is not simply connected. This implies that $\mathcal{A}^*(z_0)$ is at least doubly connected; therefore, there exists a simple closed curve $\gamma \subset \mathcal{A}^*(z_0)$ that is not null-homotopic in $\mathcal{A}^*(z_0)$. Let $D$ be the bounded component of $\widehat{\mathbb{C}} \setminus \gamma$. The fact that there is non-simple connectivity implies that $D$ must include points from the Julia set, which means $\mathcal{J}(B_{r,\alpha}) \cap D$ is not empty.

    Since the bi-accessible point $\omega_0$ lies in the Julia set, and the backward orbit of any point in the Julia set is dense in $\mathcal{J}(B_{r,\alpha})$, there exists a preimage of $\omega_0$ inside $D$. Let $\omega$ be this preimage, so $\omega \in D \cap \mathcal{J}(B_{r,\alpha})$, and let $k \ge 1$ be the smallest integer such that $B_{r,\alpha}^k(\omega) = \omega_0$.
 
    The curve $\gamma$ maps to a closed curve $\gamma' = B_{r,\alpha}^{k-1}(\gamma)$ inside the immediate basin $\mathcal{A}^*(z_0)$, and the point $\omega$ maps to $\omega' = B_{r,\alpha}^{k-1}(\omega)$. Note that $B_{r,\alpha}(\omega') = \omega_0$. Thus, $\omega'$ is a preimage of the bi-accessible point $\omega_0$, which implies $\omega' \in K$, with the set $K$ given in Lemma~\ref{lemma:path-connected-K}.
    
    By Lemma~\ref{lemma:path-connected-K}, there exists a path in $K$ connecting $\omega'$ to $\infty$. However, $\omega'$ lies in the interior of the curve $\gamma' \subset \mathcal{A}^*(z_0)$. Since the basin of attraction $\mathcal{A}^*(z_0)$ is disjoint from $K$ (which contains components of the basins of $0$ and $\infty$), the curve $\gamma'$ acts as an obstruction that the path in $K$ cannot cross.
    
    This results in a contradiction: $\omega'$ is supposed to be connected to infinity through $K$, yet it is trapped inside a bounded region $D'$ enclosed by $\gamma'$. Consequently, such a curve $\gamma$ cannot exist, implying that $\mathcal{A}^*(z_0)$ must be simply connected.
\end{proof}

\begin{theorem}
\label{thm:A-parabolica-simplemente-conexa}
    Consider $B_{r,\alpha}$ as a generalized Blaschke product, where $r \in (1, 2d+1)$, and both free critical points associated with $B_{r,\alpha}$ reside in the same Fatou component. Assume that the Julia set $\mathcal{J}(B_{r,\alpha})$ contains a bi-accessible point $\omega_0$. If $z_0 \in \mathbb{S}^1$ is a parabolic periodic point, then its immediate basin of attraction $\mathcal{A}^*(z_0)$ is simply connected. Consequently, every component of the basin of attraction $\mathcal{A}(z_0)$ is simply connected.
\end{theorem}

\begin{proof}
    We assume for the sake of contradiction that the immediate parabolic basin $\mathcal{A}^*(z_0)$ is not simply connected. Thus, there exists a simple closed curve $\gamma \subset \mathcal{A}^*(z_0)$ that is not null-homotopic within the basin. Let $D$ denote the bounded component of $\widehat{\mathbb{C}} \setminus \gamma$. Since the basin is multiply connected, $D$ must contain points belonging to the Julia set $\mathcal{J}(B_{r,\alpha})$.
    
    Given that $\omega_0$ is a bi-accessible point in the Julia set and the backward orbit of the Julia set is dense, we can find a preimage $\omega$ of $\omega_0$ such that $\omega \in D \cap \mathcal{J}(B_{r,\alpha})$. Let $k \ge 1$ be the smallest integer satisfying $B_{r,\alpha}^k(\omega) = \omega_0$.
    
    Applying the map, the curve $\gamma$ transforms into a closed curve $\gamma' = B_{r,\alpha}^{k-1}(\gamma)$ contained within the immediate basin $\mathcal{A}^*(z_0)$. Simultaneously, the point $\omega$ maps to $\omega' = B_{r,\alpha}^{k-1}(\omega)$, which satisfies $B_{r,\alpha}(\omega') = \omega_0$. This identifies $\omega'$ as a preimage of the bi-accessible point, meaning $\omega' \in K$, where $K$ is the path-connected set defined in Lemma~\ref{lemma:path-connected-K}.
    
    According to Lemma~\ref{lemma:path-connected-K}, there is a path within $K$ linking $\omega'$ to $\infty$. However, $\omega'$ is situated in the interior of $\gamma' \subset \mathcal{A}^*(z_0)$. As the parabolic basin $\mathcal{A}^*(z_0)$ is disjoint from $K$, the boundary curve $\gamma'$ serves as a topological barrier that separates $\omega'$ from the unbounded component of the complement of the basin.
    
    This leads to a contradiction: $\omega'$ must be connected to infinity via $K$, but it is enclosed by $\gamma'$. Therefore, no such non-contractible curve $\gamma$ can exist, and $\mathcal{A}^*(z_0)$ is simply connected.
\end{proof}

\subsection{Characterization of the Connectedness of Julia Sets}

We have established the necessary tools to characterize the connectedness of the Julia set for the Blaschke family. By combining the classification of Fatou components with the dynamical properties proven in previous sections (the non-existence of Siegel disks in Theorem~\ref{thm:rotation-domains} and the existence of bi-accessible cycles within the tongues in Theorem~\ref{thm:biaccessibility}) we determine exactly when the Julia set fails to be connected.

\vspace{0.5cm}

\noindent
{\small {\bf Theorem B} \textit{Let $d \in \mathbb{N}$ and consider the Blaschke family $B_{r,\alpha}$. For any parameters $r > 0$ and $\alpha \in [0,1)$, the Julia set $\mathcal{J}\left(B_{r,\alpha }\right)$ is connected if, and only if, the Fatou set $\mathcal{F}\left(B_{r,\alpha }\right)$ does not contain Herman rings.}}

\begin{proof}
    If the Julia set is connected, then all Fatou components are simply connected. Since Herman rings are doubly connected domains, they cannot exist.

    Conversely, assume that $B_{r,\alpha}$ has no Herman rings. By Theorem~\ref{thm:rotation-domains}, we know that the family does not possess Siegel disks. We demonstrate that all remaining Fatou components are simply connected by analyzing the parameter space regions established before. 

    \begin{enumerate}
        \item \textbf{The unit disk ($r \leq1$):} 
        From Lemma~\ref{lemma:Julia-is-S1-a-less-1}, the Julia set is the unit circle $\mathbb{S}^1$. Since the circle is a connected set, the result holds.

        \item \textbf{The homeomorphism/diffeomorphism region ($r \ge 2d+1$):} 
        In this region, the map is a diffeomorphism or homeomorphism on $\mathbb{S}^1$. If there are no Herman rings, then every periodic Fatou component is simply connected. Hence, the Julia set is connected.

        \item \textbf{The endomorphism region ($1 < r < 2d+1$):} 
        In this endomorphism region, we consider the possible types of Fatou components:
        \begin{itemize}
            \item The basins of attraction for $0$ and $\infty$ are simply connected by Proposition~\ref{thm:super-attracting-basins}.
            \item If there is a Fatou component $U$ that is not $\mathcal{A}(0)$ or $\mathcal{A}(\infty)$, it must be an attracting or parabolic basin associated with a cycle on $\mathbb{S}^1$. This implies that the parameter belongs to a tongue $T_{p/q}$.
            
            We distinguish two sub-cases for $U$:
            \begin{enumerate}
                \item If $U$ contains at most one free critical point, the map $B\colon U \to B(U)$ is proper of degree at most 2. The Riemann-Hurwitz formula implies that $m_U - 2 = 2(m_U - 2) + 1$, which gives $m_U = 1$, so $U$ is simply connected.
                \item If $U$ contains both free critical points, it corresponds to an adjacent parameter in the tongue. By Theorem~\ref{thm:biaccessibility}, the map possesses a bi-accessible cycle. Consequently, by Theorems \ref{thm:A-atractora-simplemente-conexa} and \ref{thm:A-parabolica-simplemente-conexa}, the immediate basin $U$ is simply connected.
            \end{enumerate}
        \end{itemize}
    \end{enumerate}
    In all cases, every Fatou component is simply connected. Therefore, the Julia set $\mathcal{J}(B_{r,\alpha})$ is connected.
\end{proof}

\subsection*{Acknowledgments}
Parts of this article have been drawn from the first author's PhD thesis, under the advise of the second author. The first author was supported by a SECIHTI (formerly CONACyT) grant No. 781230 during her graduate studies, as well as from the Sofia Kovalevskaia Award in 2022.

\printbibliography
\end{document}